\documentclass[11pt]{article}

\usepackage{amsmath,amsthm}

\usepackage{amssymb,latexsym}

\usepackage{enumerate}

\newcommand{\dyw}{\mbox{\rm div}}

\newtheorem{tw}{Theorem}[subsection]
\newtheorem{lm}[tw]{Lemma}
\newtheorem{wn}[tw]{Corollary}
\newtheorem{stw}[tw]{Proposition}
\newenvironment{dow}{\it Proof.\rm}{\hfill $\Box$}

\theoremstyle{definition}
\newtheorem*{df}{Definition}
\newtheorem{uw}[tw]{Remark}

\newcommand{\BN}{{\mathbb N}}
\newcommand{\BR}{{\mathbb R}}
\newcommand{\BX}{{\mathbb X}}

\newcommand{\WW}{{\mathcal W}}
\newcommand{\FF}{{\mathcal{F}}}
\newcommand{\GG}{{\mathcal{G}}}

\newcommand{\BB}{{\mathcal{B}}}

\newcommand{\DM}{{\mathcal{D}}}

\newcommand{\MM}{{\mathcal{M}}}

\newcommand{\PP}{{\mathcal{P}}}

\newcommand{\VV}{{\mathcal{V}}}
\newcommand{\bX}{{\mathbf{X}}}
\newcommand{\tdw}{{\tilde{w}}}
\newcommand{\bfX}{{\mathbf{X}}}
\newcommand{\ovV}{{\overline{v}}}

\newcommand{\BRD}{{\mathbb{R}^{d}}}

\newcommand{\essinf}{\mathop{\mathrm{ess\,inf}}}
\newcommand{\esssup}{\mathop{\mathrm{ess\,sup}}}

\newcommand{\nsubsection}{\setcounter{equation}{0}\subsection}

\begin{document}

\title {Obstacle problem for semilinear parabolic equations
with measure data}
\author {Tomasz Klimsiak and Andrzej
Rozkosz
\smallskip\\
{\small Faculty of Mathematics and Computer Science,
Nicolaus Copernicus University} \\
{\small  Chopina 12/18, 87--100 Toru\'n, Poland}\\
}
\date{}
\maketitle
\begin{abstract}
We consider the obstacle problem with two irregular  barriers for
the Cauchy-Dirichlet problem for semilinear parabolic equations
with measure data. We prove the existence and uniqueness of
renormalized solutions of the problem and well as results on
approximation of the solutions by the penaliztion method. In the
proofs we use probabilistic methods of the theory of Markov
processes and the theory of backward stochastic differential
equations.
\end{abstract}



\nsubsection{Introduction}

Let $D\subset\mathbb {R}^{d}$, $d\ge2$, be an open bounded set and
let $\mu$ be a bounded soft measure on $D_{T}:=(0,T]\times D$ (we
call a Radon measure $\mu$ soft if it does not charge sets of zero
parabolic capacity). Suppose we are also given
$f:D_T\times\BR\rightarrow\BR$, $\varphi:D\rightarrow\BR$ and two
functions $h_1,h_2:D_T\rightarrow\bar{\mathbb{R}}$ such that
$h_1\le h_2$. In the present paper we investigate the obstacle
problem
\begin{equation}
\label{eq1.1}
\left\{
\begin{array}{l}\frac{\partial u}{\partial t}+A_{t}u\le-f_u-\mu
\mbox{ on } \{u\ge h_{1}\},\smallskip \\
\frac{\partial u}{\partial t}+A_{t}u\ge-f_u-\mu
\mbox{ on } \{u\le h_{2}\},\smallskip\\
h_{1}\le u\le h_{2} \mbox{ on }D_T\,, \smallskip\\
u(T, \cdot)=\varphi, \quad u(t, \cdot)_{|\partial D}=0, \quad t
\in(0,T).
\end{array}
\right.
\end{equation}
Here $f_u(t,x)=f(t,x,u(t,x))$, $(t,x)\in D_T$, and
\begin{equation}
\label{eq1.07} A_{t}=\frac 12 \sum_{i,j=1}^{d}
\frac{\partial}{\partial
x_{j}}(a_{ij}(t,x)\frac{\partial}{\partial x_{i}}),
\end{equation}
where $a: D_{T}\rightarrow\BR^d\otimes\BR^d$ is a measurable
symmetric matrix-valued function  such that for some
$\Lambda\ge1$,
\begin{equation}
\label{eq1.002} \Lambda^{-1}|y|^2\le\sum_{i,j=1}^{d} a_{ij}(t,x)
y_{i}y_{j}\le\Lambda |y|^{2},\quad y\in \mathbb{R}^{d}
\end{equation}
for a.e. $(t,x)\in D_{T}$.

Problem (\ref{eq1.1}) with regular barriers  and $L^2$ data is
quite well investigated (see, e.g., the classical monograph
\cite[Section 3.2]{BL}). There are only few papers devoted to problem
(\ref{eq1.1}) with one irregular time-dependent  barrier and to
our knowledge there is no paper on problem (\ref{eq1.1}) with two
irregular barriers and general soft measure on the right-hand
side. In \cite{MP,Pierre1} the linear problem with one barrier and
$L^2$ data is considered. A semilinear problem (\ref{eq1.1}) with
$L^2$ data and $f$ satisfying a Lipschitz and a linear growth
condition is considered in \cite{Kl:SPA} in the case of one
barrier and in \cite{Kl:PA2} in the case of two barriers. Note,
however, that unlike the present paper, in \cite{Kl:SPA,Kl:PA2}
the function $f$ may depend on the solution of (\ref{eq1.1}) as
well as its gradient.

The main goal of the paper is to prove the existence and
uniqueness of solutions of (\ref{eq1.1}) in the case where the
barriers $h_1,h_2$ are merely measurable and satisfy some kind of
separation condition, $\varphi\in L^1(D)$, $f$ satisfies a
monotonicity condition in $u$ and some mild growth condition
considered earlier in the theory of PDEs with measure data (see,
e.g., \cite{BBGGPV}). We are also interested in the approximation
of solutions of (\ref{eq1.1}) by the penalization method. As in
\cite{Kl:SPA,Kl:PA2}, to study these problems we adopt a
stochastic approach based on the theory of backward stochastic
differential equations.

The problem of existence and uniqueness of solutions of
(\ref{eq1.1}) with measure data is delicate even for regular
barriers. If the barriers are irregular then the additional
serious problem is to define solutions of (\ref{eq1.1})  in  a way
that ensures their uniqueness. The classical approach via
variational inequalities is not suitable even in the case of one
barrier and $L^2$ data, because even in that case the solution of
the variational inequality with irregular barrier is in general
not unique.

To overcome the difficulty with uniqueness of solutions one could
try to adapt to the parabolic case the method of minimal
solutions, which was applied successfully in \cite{Leone} to
one-sided elliptic obstacle problems with general Radon measure on
the right-hand side. Unfortunately, the concept of minimal
solutions is not directly applicable to the two-sided obstacle
problem.

To address the nonuniqueness problem one can also try to adopt the
approach from the paper \cite{Pierre1} devoted to linear parabolic
one-sided obstacle problems with $L^2$ data and define a solution
of (\ref{eq1.1}) as a pair $(u,\nu)$ consisting of a measurable
function $u:D_T\rightarrow\BR$ and a soft measure $\nu$ on $D_T$
such that
\begin{equation}
\label{eq1.4} \left\{
\begin{array}{l}\frac{\partial u}{\partial t}+A_{t}u
=-f_{u}-\mu-\nu\mbox{ in }D, \smallskip \\
u(T,\cdot)=\varphi, \quad u(t,\cdot)_{|\partial D}=0, \quad
t\in(0,T)
\end{array}
\right.
\end{equation}
is satisfied in some weak sense, $h_1\le u\le h_2$ a.e.  and $\nu$
satisfies some minimality condition. In case $h_1,h_2$ and $u$ are
regular, the natural minimality condition says that
\begin{equation}
\label{eq1.05}
\int_{D_T}(u-h_1)\,d\nu^+=\int_{D_T}(h_2-u)\,d\nu^-=0,
\end{equation}
where $\nu^+,\nu^-$ denote the positive and the negative part of
the Jordan decomposition of $\nu$, i.e. the reaction measure acts
only when $u$ touches the barriers. If $h_1,h_2$ are merely
measurable, the problem is to make sense of (\ref{eq1.05}),
because in general $\nu$ is not absolutely continuous with respect
to the Lebesgue measure. A further complication arises from the
fact that in general $u$ is not quasi-continuous. It is, however,
quasi-l.s.c., so determined q.e. Therefore one can hope that when
the barriers are quasi-l.s.c. or quasi-u.s.c., and hence
determined q.e., then the minimality condition holds (the
integrals in (\ref{eq1.05}) are then well defined since $\nu$ is
soft). Unfortunately, even in that case the integrals in
(\ref{eq1.05}) may be strictly positive (a simple example is to be
found  in \cite{Kl:SPA}).

Our definition of a solution of (\ref{eq1.1}) is based on a
nonlinear Feynman-Kac formula.
It may be viewed as a probabilistic extension 
of the definition described above, because in the linear case with
$L^2$ data our probabilistic solution $(u,\nu)$ coincides with the
solution considered in \cite{Pierre1}. Let us also note that in
the case of one-sided problem the first component $u$ of the
probabilistic solution is a minimal solution in the sense that
\begin{equation}
\label{eq1.06} u=\mbox{\rm quasi-essinf}\{v\ge
h_{1},\,m_{1}\mbox{-a.e.}:\,\, v\mbox{ is a supersolution of
problem \,\rm{(\ref{eq1.04})}}\},
\end{equation}
where $m_1$ is the Lebesgue measure on $\BR_+\times\BR^d$ and
\begin{equation}
\label{eq1.04} \left\{
\begin{array}{l}\frac{\partial u}{\partial t}+A_{t}u
=-f_{u}-\mu\mbox{ in }D, \smallskip \\
u(T,\cdot)=\varphi, \quad u(t,\cdot)_{|\partial D}=0, \quad
t\in(0,T).
\end{array}
\right.
\end{equation}
Thus our probabilistic approach leads to a parabolic analogue of a
solution considered in \cite{Leone}. In the case of two-sided
problem it leads to some variant of (\ref{eq1.06}).

Let $\BX=(X,P_{s,x})$ be a Markov family associated with the
operator $A_t$ and for a soft measure $\gamma$ on $D_T$ let
$A^{\gamma}$ denote the additive functional of $\BX$ associated
with $\gamma$ in the Revuz sense (see Section \ref{sec2}). Set
\begin{equation}
\label{eq1.09}
\xi^s=\inf\{t\ge s:X_t\notin D\}.
\end{equation}
By a solution of (\ref{eq1.1}) we mean a pair $(u,\nu)$ consisting
of a function $u$ on $D_T$ and a bounded soft measure $\nu$ on
$D_T$ such that
\begin{equation}
\label{eq1.2} u(s,x)=
E_{s,x}\Big(\mathbf{1}_{\{\xi^s>T\}}\varphi(X_T)
+\int_{s}^{\xi^s\wedge T}f_{u}(t,X_t)\,dt +\int_{s}^{\xi^s\wedge
T} d(A_{s,t}^{\mu}+A_{s,t}^{\nu})\Big)
\end{equation}
for q.e. $(s,x)\in D_T$ and $\nu$  satisfies a minimality
condition introduced in \cite{Kl:SPA}. The minimality condition
says that for any measurable functions $h_{1}^{*}, h_{2}^{*}$ on
$D_T$ having the property that $h_{1}\le h_{1}^{*}\le u\le
h_{2}^{*}\le h_{2}$ a.e. and the processes $[s,T]\ni t\mapsto
h_i^{*}(t,X_t)$, $i=1,2$, are c\`adl\`ag under $P_{s,x}$ for q.e.
$(s,x)\in D_T$ the following equalities
\begin{equation}
\label{eq1.3} \int_{s}^{\xi^s\wedge T}
(u_{-}(t,X_t)-h_{1-}^{*}(t,X_t))\,dA_{s,t}^{\nu^{+}}
=\int_{s}^{\xi^s\wedge T}
(h_{2-}^{*}(t,X_t)-u_{-}(t,X_t))\,dA_{s,t}^{\nu^{-}}=0
\end{equation}
hold $P_{s,x}$-a.s. for q.e. $(s,x)\in D_{T}$ (In (\ref{eq1.3}),
$g_{-}(t,X_t)=\lim_{s<t,s\rightarrow t}g(s,X_s)$ for $g:=
u,h^{*}_1,h^{*}_2$).

Our main result says that under natural mild  assumptions on the
data there exists a unique solution $(u,\nu)$ of (\ref{eq1.1}).
Moreover, from the nonlinear Feynman-Kac formula (\ref{eq1.2}) we
deduce that $u$ is a renormalized (and an entropy as well)
solution of problem (\ref{eq1.4}). If $A^{\mu}$ is continuous and
$h_1,h_2$ are quasi-continuous, condition (\ref{eq1.3}) has purely
analytic interpretation. We show that under this additional
assumption $u$ is quasi-continuous and (\ref{eq1.3}) reduces to
condition (\ref{eq1.05}). In the case of irregular barriers an
analytical formulation of the minimality condition is possible in
the linear case with $L^2$ data. In that case (\ref{eq1.3}) may be
expressed by using the notion of precise version of function
introduced in \cite{Pierre}.

The reason for adopting here a probabilistic definition of a
solution of (\ref{eq1.1}) not only pertains to  the difficulties
with the analytic formulation of the minimality condition for the
reaction measure $\nu$. One major  advantage of the probabilistic
definition is that it fits well to a general scheme of proving
existence of solutions of equations with measure data which was
successfully adopted in the paper \cite{KR} devoted to semilinear
elliptic equations with operators associated with general
symmetric regular Dirichlet forms. The scheme comprises two
essentially different parts. In the present context the first part
consists in using  stochastic methods to show the existence of a
pair $(u,\nu)$ satisfying (\ref{eq1.2}), (\ref{eq1.3}) and such
that the process $t\mapsto u(t,X_t)$ has some integrability
properties. As a matter of fact this part follows  rather easily
from results on doubly reflected BSDEs proved recently in
\cite{Kl:BSM}.  The second part consists in using the nonlinear
Feynman-Kac formula (\ref{eq1.2}) to prove additional regularity
properties of $u$ and to show that $u$ is a renormalized solution
of (\ref{eq1.4}).

Results from \cite{Kl:BSM} are also used to show that the solution
of (\ref{eq1.1}) can be approximated by the penalization method.
For instance, we show that under the same assumptions under which
there exists a unique solution $(u,\nu)$ of (\ref{eq1.1}), if
$u_{n}$ is a renormalized solution of the problem
\[
\left\{
\begin{array}{l}\frac{\partial u_{n}}{\partial t}+A_{t}u_{n}
=-f_{u_{n}}-\mu-n(u_{n}-h_{1})^{-}+n(u_{n}-h_{2})^{-}, \smallskip\\
u_{n}(T, \cdot)=\varphi, \quad u_{n}(t, \cdot)_{|\partial D}=0,
\quad t\in(0,T),
\end{array}
\right.
\]
then $u_{n}\rightarrow u$ q.e. on $D_{T}$ and  $\nabla
u_{n}\rightarrow\nabla u$ a.e. From \cite{Kl:BSM} and the main
result of the present paper it also follows that q.e. on $D_T$,
\begin{align}
\label{eq1.11} \nonumber u=\mbox{quasi-essinf}&\mbox{$\{v\ge h_1,
m_1$-a.e.: $v$ is a supersolution }\\
&\quad\mbox{of (\ref{eq1.04}) with $\mu$ replaced by
$\mu-\nu^{-}$}\}.
\end{align}
If $h_2\equiv+\infty$ then $\nu^-=0$, and so (\ref{eq1.11})
reduces to (\ref{eq1.06}). Finally, let us mention that from
\cite{Kl:BSM} and our main result it follows that $u$ can also be
characterized as a solution of the following stopping time problem
(sometimes called Dynkin game): for any $h_{1}^{*}, h_{2}^{*}$ as
in condition (\ref{eq1.3}),
\begin{align*}
u(s,x)&=\sup_{\sigma\in\mathcal{T}^s}\inf_{\delta\in\mathcal{T}^s}
E_{s,x} \Big(\int_s^{\sigma\wedge\delta\wedge T}
f_{u}(t,X_t)\,dt
+\int_{s}^{\sigma\wedge\delta\wedge T}dA^{\mu}_{s,t}\\
&\quad +h^{*}_{1}(\delta,X_{\delta})
\mathbf{1}_{\{\delta\le\sigma<T\}}\mathbf{1}_{\{\delta<\xi^s\wedge
T\}} +h^{*}_{2}(\sigma,X_{\sigma})\mathbf{1}_{\{\sigma<\delta\}}
\mathbf{1}_{\{\sigma<\xi^s\wedge T\}}\\
&\quad +\varphi(X_T)\mathbf{1}_{\{\sigma=\delta=T\}}\Big)
\end{align*}
for q.e. $(s,x)\in D_T$, where $\mathcal{T}^s$ denotes the set of
all stopping times with values in $[s,T]$ with respect to the
completion of the filtration generated by $X$.

In the present paper we are mostly interested in the investigation
of renormalized (or entropy) solutions of (\ref{eq1.1}). But it is
worth mentioning that as a byproduct of our proofs we obtain new
results on stochastic representation of solutions of (\ref{eq1.1})
and the Cauchy-Dirichlet problem (\ref{eq1.04}), which can be
regarded as problem (\ref{eq1.1}) with $h_1\equiv-\infty$,
$h_2\equiv+\infty$. Some of these results seem to be new even in
the case of problem (\ref{eq1.04}) with $L^2$ data. To our
knowledge in all the existing results on stochastic representation
of solutions to that problem some regularity of the boundary of
the domain is assumed. In the present paper we do not require any
regularity of $D$. Let us also note that in the recent paper
\cite{PPP} the existence and uniqueness of renormalized solutions
of the Cauchy-Dirichlet problem with more general than
(\ref{eq1.07}) divergence form operator (for instance $p$-Laplace
operator) and $f$ not depending on $x$ and satisfying the
so-called sign condition is proved. In our paper we consider
equations with $A$ given by (\ref{eq1.07}) but we  allow $f$ to
depend on $x$.

\nsubsection{Preliminaries}
\label{sec2}

In this section we have compiled some basic facts on diffusions
associated with the operator $A_t$ defined by (\ref{eq1.07}) and
their additive functionals associated with soft measures on
$\BR_+\times\BR^d$. Here and in the next sections we assume that
$a$ satisfies (\ref{eq1.002}). By putting $a_{ij}=\delta^i_j$
outside $D_T$ we can and will assume that $a$ is defined and
satisfy (\ref{eq1.002}) in all $\BR_+\times\BR^d$.

\subsubsection{Time-inhomogeneous diffusions and additive functionals}

Let $\Omega=C(\BR_+;\BR^d)$ be the space of  continuous
$\BR^d$-valued functions on $\BR_+=[0,+\infty)$, $X$ be the
canonical process on $\Omega$, $\FF^s_t=\sigma(X_u,u\in[s,t])$ and
for given $T>0$ let $\bar\FF^s_t=\sigma(X_u,u\in[T+s-t,T])$. We
define $\GG^s_T$ as the completion of $\FF^s_T$ with respect to
the family $\PP=\{P_{s,\mu}:\mu$ is a probability measure on
$\BB(\BR^d)$$\}$, where
$P_{s,\mu}(\cdot)=\int_{\BR^d}P_{s,x}(\cdot)\,\mu(dx)$, and then
we define $\GG^s_t$ ($\bar\GG^s_t$) as the completion of $\FF^s_t$
($\bar\FF^s_t$) in $\GG^s_T$ with respect to $\PP$.

Let $p$ denote the fundamental solution for the operator $A_t$ and
let $\BX=\{(X, P_{s,x}):(s,x)\in\BR_+\times\BR^d\}$ be a
time-inhomogeneous Markov process for which $p$ is the transition
density function, i.e.
\[
P_{s,x}(X_t=x;0\le t\le s)=1,\quad
P_{s,x}(X_t\in\Gamma)=\int_{\Gamma}p(s,x,t,y)\,dy,\quad t>s
\]
for any $\Gamma\in\BB(\BR^d)$. The process $\BX$ admits the
so-called strict Fukushima decomposition, i.e. for every $T>0$,
\begin{equation}
\label{eq1.03} X_{t}= X_{s}+A_{s,t}+M_{s,t},\quad t\in[s,T], \quad
P_{s,x}\mbox{-a.s.}
\end{equation}
for every $(s,x)\in[0,T]\times\BR^d$, where $\{A_{s,t},0\le s\le
t\le T\}$ is a two parameter continuous additive functional (CAF)
of $\BX$ (with respect to the filtration $\{\GG^s_t\}$) of zero
energy and $\{M_{s,t}, 0\le s\le t\le T\}$ is a two parameter
martingale additive functional (MAF) of $\BX$ (with respect to
$\{\GG^s_t\}$) of finite energy (see \cite{R:PTRF}). In
particular, $M_{s,\cdot}$ is a $(\{\GG^s_t\},P_{s,x})$-martingale.
Moreover, for every $(s,x)\in[0,T]\times\BR^d$,
\begin{equation}
\label{eq1.6} \langle M_{s,\cdot}^{i}, M_{s,\cdot}^{j}\rangle_{t}
=\int_{s}^{t}a_{ij}(\theta, X_{\theta})\,d\theta,\quad t\in[s,T],
\quad P_{s,x}\mbox{-a.s.}
\end{equation}
(see \cite{R:PTRF}). It follows that under $P_{s,x}$ the process
$B_{s,\cdot}$ defined as
\[
B_{s,t}=\int_{s}^{t} \sigma^{-1}(\theta, X_{\theta})\,dM_{s,
\theta},\quad t\in[s,T],
\]
where $\sigma\cdot\sigma^{*}=a$ and $\sigma^{-1}$ is the inverse
matrix of $\sigma$, is a Brownian motion on $[s,T]$. It is also
known (see \cite{K:JTP}) that $\BX$ admits the so-called
Lyons-Zheng decomposition, i.e. for any $T>0$,
\[
X_{t}-X_{u}= \frac 1 2 M_{u,t}+\frac12(N^{s,x}_{T+s-t}-N_{s,
T+s-u}^{s,x}) -\alpha^{s,x}_{u,t}, \quad s\le u \le t \le T
\]
for every $(s,x)\in[0,T]\times\mathbb{R}^{d}$, where
$N_{s,\cdot}^{s,x}$ is an $(\{\bar G^s_t\},P_{s,x})$-martingale
and
\[
\alpha^{s, x, i}_{u,t}= \sum_{j=1}^{d} \int_{u}^{t}\frac12a_{ij}
(\theta, X_{\theta}) p^{-1} \frac{\partial p}{\partial y_{j}}(s,x,
\theta, X_{\theta})\,d\theta.
\]

For a measurable $\mathbb{R}^{d}$-valued function $\mathbf{f}$ and
$s\le r\le t\le T$ we put
\begin{align*}
\int_{r}^{t} \mathbf{f}(\theta, X_{\theta})\,d^{*}X_{\theta} & :=
-\int_{r}^{t}\mathbf{f}(\theta,X_{\theta})\,(dM_{s,\theta}
+d\alpha_{s,\theta}^{s,x}) \\
&\quad -\int_{T+s-t}^{T+s-r}\mathbf{f}(T+s-\theta,
X_{T+s-\theta})\,dN^{s,x}_{s,\theta}.
\end{align*}
The usefulness of the backward integral defined above comes from
the fact that for regular $\mathbf{f}$,
\[
\int_{u}^{t}\mbox{div}\mathbf{f}(\theta, X_{\theta})\,d\theta
=\int_{u}^{t} a^{-1} \mathbf{f}(\theta,
X_{\theta})\,d^{*}X_{\theta}, \quad s \le u \le t \le T, \quad
P_{s,x} \mbox{-a.s.},
\]
where $a^{-1}$ is the inverse matrix of $a$ (see, e.g.,
\cite{R:Stochastics}).

Let us recall that for every $\Phi\in
L^{2}(0,T;H^{-1}(\mathbb{R}^{d}))$ there exist $f\in
L^{2}([0,T]\times\BR^{d})$ and $G=(G^1,\dots,G^d)$ with $G^i\in
L^{2}([0,T]\times\BR^{d})$, $i=1,\dots,d$, such that
\[
\Phi= f+\mbox{div}(G),
\]
i.e. for every $\eta\in L^{2}(0,T;H^{1}(\mathbb{R}^{d}))$,
\[
\Phi(\eta)=\int_{[0,T]\times\BR^d}\eta
f\,dm_{1}-\int_{[0,T]\times\BR^d}G\cdot \nabla \eta\,dm_{1},
\]
where $m_1$ denote the Lebesgue measure on $\BR_+\times\BR^d$.

The following lemma will be needed in the next section to prove
regularity of parabolic potentials. For the meaning of the notion
of ``quasi every" used below see Section \ref{sec2.3}.

\begin{lm}
\label{lm1.1} Assume that $\Phi_{1},\Phi_{2}\in L^{2}(0,T;
H^{-1}(\mathbb{R}^{d}))$, $D\subset \mathbb{R}^{d}$ is an open
bounded set and $\Phi_{1}=\Phi_{2}$ in $L^2(0,T; H^{-1}(D))$. If
$\Phi_{1}=f_{1}+\dyw(G_{1})$, $\Phi_{2}=f_{2}+\dyw(G_{2})$ then
for quasi every  $(s,x)\in D_{T}$,
\begin{align*}
&\int_{s}^{t\wedge\xi^{s}} f_{1}(\theta, X_{\theta})\,d\theta +
\int_{s}^{t\wedge\xi^{s}}a^{-1}G_{1}(\theta,
X_{\theta})\,d^{*}X_{\theta}\\
&\qquad=\int_{s}^{t \wedge\xi^{s}} f_{2}(\theta,
X_{\theta})\,d\theta+ \int_{s}^{t \wedge\xi^{s}}
a^{-1}G_{2}(\theta, X_{\theta})\,d^{*}X_{\theta},\quad t\in[s,T),
\quad P_{s,x}\mbox{\rm-a.s.}.
\end{align*}
\end{lm}
\begin{dow}
Let $D_{\varepsilon}=\{x\in D:\mbox{dist}(x, D^{c})>
\varepsilon\}$ and let $\Phi_{i}^{\varepsilon}=f_{i}^{\varepsilon}
+\dyw (G_{i}^{\varepsilon})$, $i=1,2$, where $f^{\varepsilon}_{i},
G^{\varepsilon}_{i}$ are standard regularizations of $f_{i},
G_{i}$, respectively. It is an elementary check that
\[
\Phi_{1}^{\varepsilon}=\Phi_{2}^{\varepsilon}\mbox{ on }
L^{2}(0,T;H^{-1}(D_{\varepsilon})).
\]
Since $\Phi_{1}^{\varepsilon},\Phi_{2}^{\varepsilon}\in
L^{2}([0,T]\times\BR^d)$, it follows from the above equality that
\[
\Phi_{1}^{\varepsilon}= \Phi_{2}^{\varepsilon}, \quad
m_{1}\mbox{-a.e.}\mbox{ on }[0,T] \times D_{\varepsilon}.
\]
In \cite{K:JTP} it is proved that
\[
\int_{s}^{\cdot} f_{i}^{\varepsilon}(\theta, X_{\theta})\,d\theta
+\int_{s}^{\cdot} a^{-1}G_{i}^{\varepsilon} (\theta,
X_{\theta})\,d^{*}X_{\theta}\rightarrow
\int_{s}^{\cdot}f_{i}(\theta, X_{\theta})\, d \theta
+\int_{s}^{\cdot} a^{-1}G_{i}(\theta, X_{\theta})\,d^{*}X_{\theta}
\]
uniformly on $[s,T]$ in probability $P_{s,x}$ for q.e. $(s,x)\in
[s,T]\times\BR^d$. Let $\xi^{s,\varepsilon}=\inf\{t\ge s,
X_{t}\notin D_{\varepsilon}\}$. Then
\begin{align*}
&\int_{0}^{t\wedge\xi^{s,\varepsilon}}f_{1}^{\varepsilon}(\theta,
X_{\theta})\,d\theta +\int_{0}^{t\wedge\xi^{s,\varepsilon}}
a^{-1}G^{\varepsilon}_{1}(\theta, X_{\theta})\,d^{*}X_{\theta}
=\int_{0}^{t\wedge\xi^{s,\varepsilon}}\Phi_{1}^{\varepsilon}(\theta,
X_{\theta})\,d\theta \\
&\quad =
\int_{0}^{t\wedge\xi^{s,\varepsilon}}\Phi_{2}^{\varepsilon}(\theta,
X_{\theta})\,d\theta=
\int_{0}^{t\wedge\xi^{s,\varepsilon}}f_{2}^{\varepsilon}(\theta,
X_{\theta})\,d\theta +
\int_{0}^{t\wedge\xi^{s,\varepsilon}}a^{-1}G_{2}^{\varepsilon}(\theta,
X_{\theta})\,d^{*}X_{\theta},
\end{align*}
from which the desired result follows, because
$\xi^{s,\varepsilon}\nearrow\xi^{s}$, $P_{s,x}$-a.s. as
$\varepsilon \searrow 0$.
\end{dow}

\subsubsection{Time-homogeneous diffusions}
\label{sec2.2}

In the next sections it will be advantageous to consider a certain
time-homogeneous diffusion determined by $\BX$. The standard
construction of it is as follows. We set
\begin{equation}
\label{eq2.07} \Omega'=\BR_+\times\Omega,\quad
P'_{s,x}(B)=P_{s,x}(\{\omega\in\Omega:(s,\omega)\in B\})
\end{equation}
and consider the proces $\mathbf{X}$ on $\Omega'$ defined as
\begin{equation}
\label{eq2.09}
\mathbf{X}_{t}(s,\omega)=(s+t,X_{s+t}(\omega)),\quad t\ge0.
\end{equation}
Let $\FF'_t=\sigma(\mathbf{X}_u,u\le t)$,
$\FF'_{\infty}=\sigma(\mathbf{X}_u,u<+\infty)$ and let
$\GG'_\infty$ denote the completion of $\FF'_{\infty}$ with
respect to the family $\PP'=\{P'_{\mu}:\mu$ is a probability
measure on $\BR_+\times\BR^d\}$, where
$P'_{\mu}(\cdot)=\int_{\BR_+\times\BR^d}P'_{s,x}(\cdot)\mu(ds\,dx)$.
Finally, let $\GG'_t$ denote the completion of $\FF'_t$ in
$\GG'_{\infty}$ with respect to $\PP'$. Then
$\BX'=\{(\mathbf{X}_{t},P'_{s,x});(s,x)\in\BR_+\times\BR^d\}$ is a
time-homogeneous Markov process with respect to the filtration
$\{\GG'_t\}$ with the transition density
\begin{equation}
\label{eq2.08} P'(t,(s,x),\Gamma)=P(s,x,s+t,\Gamma_{s+t}),
\end{equation}
where $\Gamma_{s+t}=\{x\in\BR^d:(s+t,x)\in\Gamma\}$.

For $h,t\ge0$ we define $\theta_h:\Omega'\rightarrow\Omega'$ and
$\tau(t):\Omega'\rightarrow\BR_+$ by putting
\[
\theta_h\omega'=(s+h,\omega), \quad
\tau(t)(\omega')=s+t=\tau(0)(\omega')+t
\]
for $\omega'=(s,\omega)$. Then
\[
\mathbf{X}_{t}(\theta_h\omega')=\mathbf{X}_{t+h}(\omega'),\quad
\tau(t)\circ\theta_h=\tau(t)+h,\quad h,t\ge0.
\]
Since the decomposition (\ref{eq1.03}) is unique, we may assume
that the two parameter AFs $A,M$ of (\ref{eq1.03}) are defined for
all $0\le s\le t$ and that for every $(s,x)\in\BR_+\times\BR^d$
the decomposition (\ref{eq1.03}) holds true for all $t\ge s$. Set
\begin{equation}
\label{eq2.007} A_t(\omega')=A_{s,s+t}(\omega),\quad
M_t(\omega')=M_{s,s+t}(\omega),\quad t\ge0
\end{equation}
for $\omega'=(s,\omega)$. Since
$P_{s,x}(A_{s,t}=A_{s,u}+A_{u,t},0\le s\le u\le t)=1$ for every
$(s,x)\in\BR_+\times\BR^d$ and $M$ has the same property, it
follows from (\ref{eq2.007}) and the fact that the energy of $A$
equals zero and the energy of $M$ is finite that $A=\{A_t,t\ge0\}$
is a CAF of $\BX'$ of zero energy and $M=\{M_t,t\ge0\}$ is a MAF
of $\BX'$ of finite energy. In particular, $M$ is a
$(\{\GG'_t\},P'_{s,x})$-martingale for every
$(s,x)\in\BR_+\times\BR^d$. Moreover, from (\ref{eq1.03}) and
(\ref{eq2.09}) it follows that for every
$(s,x)\in\BR_+\times\BR^d$,
\[
X_{\tau(t)}=X_{\tau(0)} + A_t + M_t, \quad t \ge 0 \quad P'_{s,x}
\mbox{-a.s.}
\]
(By convention, if $\xi$ is a random variable defined on $\Omega$
we set $\xi(\omega')=\xi(\omega)$ for any
$\omega'=(s,\omega)\in\Omega'$. Thus, in particular,
$X_{\tau(t)}(\omega')=X_{\tau(t)(\omega')}(\omega)$).

Set
\begin{equation}
\label{eq1.8}
B_{t}=\int_{0}^{t}\sigma^{-1}(\mathbf{X}_{\theta})\,dM_{\theta},
\quad t\ge 0.
\end{equation}
By (\ref{eq1.6}) and (\ref{eq2.007}),
\[
\langle M^i,M^j\rangle_t=\langle
M^i_{s,\cdot},M^j_{s,\cdot}\rangle_{s+t}
=\int^{s+t}_sa_{ij}(\theta,X_{\theta})\,d\theta, \quad t\ge0,\quad
P'_{s,x}\mbox{-a.s.}
\]
for every $(s,x)\in\BR_+\times\BR^d$. Therefore $B$ is a Brownian
motion under $P'_{s,x}$ for every $(s,x)\in\BR_+\times\BR^d$. In
fact, by \cite[Theorem 12]{Lejay}, it is an $\{\GG'_t\}$-Brownian motion.

\subsubsection{Capacity and soft measures}
\label{sec2.3}

Let $\WW$ be the space of  $u\in L^2(\BR_+;H^1_0(\BRD))$ such that
$\frac{\partial u}{\partial t}\in L^2(\BR_+;H^{-1}(\BRD))$ endowed
with the usual norm $\|u\|_{\WW}=\|u\|_{L^2(\BR_+;H^1_0(\BRD))}
+\|\frac{\partial u}{\partial t}\|_{L^2(\BR_+;H^{-1}(\BRD))}$. We
define the parabolic capacity of an open set
$U\subset\BR_+\times\BR^d$ as
\[
\mbox{cap}(U)=\inf\{\|u\|_{\WW}:u\in\WW,u\ge\mathbf{1}_U\mbox{
a.e. in } \BR_+\times\BR^d\}.
\]
The parabolic capacity of a Borel set $B\subset\BR_+\times\BR^d$
is defined as
\[
\mbox{cap}(B)=\inf\{\mbox{cap}(U):U\mbox{ is an open subset of }
\BR_+\times\BR^d,B\subset U\}.
\]
It is known that for a Borel $B\subset\BR_+\times\BR^d$,
$\mbox{cap}(B)=0$ iff
\[
\int_{\BR_+}P_{s,m}(\exists t>s:(t,X_t)\in B)\,ds=0,
\]
where $P_{s,m}(\cdot)=\int_{\BRD}P_{s,x}(\cdot)\,dx$ and $m$
denotes the Lebesgue measure on $\BR^d$ (see the argument
following Eq. (5.2) in \cite{Oshima2}).

We will say that some property is satisfied quasi-everywhere (q.e.
for short) if it is satisfied except for some Borel subset of
$\BR_+\times\BR^d$ of capacity zero.

Let $\mu$ be a signed Borel  measure on $\BR_+\times\BR^d$ and let
$|\mu|$ denote the total variation of $\mu$.  We say that $\mu$ is
smooth if $|\mu|(B)=0$ for every Borel set
$B\subset\BR_+\times\BR^d$ such that $\mbox{cap}(B)=0$ and there
exists an ascending sequence $\{F_{n}\}$ of closed subsets of
$\BR_+\times\BR^d$ such that for every compact
$K\subset\BR_+\times\BR^d$, $\lim_{n\rightarrow
\infty}\mbox{cap}(K\setminus F_{n})\rightarrow 0$,
$|\mu|(F_{n})<\infty$, $n\ge 1$, and
$\mu(\BR_+\times\BR^d\setminus \bigcup F_{n})=0$. We say that a
smooth measure $\mu$ is soft if $\mu$ is a Radon measure. We say
that $\mu$ is a soft measure on $V\subset\BR_+\times\BR^d$ if it
is soft and $\mu(\BR_+\times\BR^d\setminus V)=0$.

In what follows by $S$ we will denote the set of all positive
smooth measures on $\BR_+\times\BR^d$ and by $S_0$ the subset of
$S$ consisting of all measures of finite energy integral (see,
e.g., \cite[Section 5]{Oshima} for the definition). By
$\MM_b(D_T)$ we denote the set of all signed Borel measures on
$\BR_+\times\BR^d$  having bounded variation and such that
$\mu(\BR_+\times\BR^d\setminus D_T)=0$. By $\MM^+_{b}(D_T)$ we
denote the subset of $\MM_{b}(D_T)$ consisting of all positive
measures. We write $\|\mu\|_{TV}$ for the total variation norm of
$\mu\in\MM_b(D_T)$. By $\MM_{0}(D_T)$ we denote the set of all
soft measures on $D_T$, by $\MM_{0,b}(D_T)$ the set of all bounded
soft measures on $D_T$ and by $\MM^+_{0,b}(D_T)$ the subset of
$\MM_{0,b}(D_T)$ consisting of all positive measures.

Let us note that in the literature soft measures are usually
defined on $(0,T)\times D$. In the present paper we consider soft
measures on $(0,T]\times D$ because in general the obstacle
reaction measure is concentrated on that set.

\subsubsection{Additive functionals and soft measures}

Let $E_{s,x}$ (resp. $E'_{s,x}$) denote the expectation with
respect to $P_{s,x}$ (resp. $P'_{s,x}$). Let us recall that a
positive additive functional $A$ of $\BX'$ and a positive soft
measure $\mu$ on $\BR_+\times\BR^d$ are in the Revuz
correspondence if
\begin{equation}
\label{Revuz} \int_{\BR^+\times\BR^d} f\,d\mu=\lim_{\alpha\rightarrow
\infty}\alpha \int_{\BR_+\times\BR^d}(E'_{s,x}\int_{0}^{\infty}
e^{-\alpha t} f(\mathbf{X}_{t})\,dA_{t})\,dm_1(s,x)
\end{equation}
for every $f\in\BB^+(\BR_+\times\BR^d)$.  If
$\langle\mu,1\rangle<+\infty$ then $A$ is called integrable. By
Theorems 6.4.7 and 6.4.9 in \cite{Oshima1}, for every $\mu\in
S_{0}$ there exists a unique positive natural additive functional
(NAF) $A$  of $\mathbb{X}'$ in the Revuz correspondence with $\mu$
(see also (5.9), (5.10) in \cite{Oshima}). Since each measure
$\mu\in S$ may be approximated by measures in $S_{0}$ (see
\cite[Theorem 5.6]{Oshima}), repeating step by step the proof of
\cite[Theorem 4.1.16]{Oshima1} (see also \cite[Theorem
5.6]{Oshima}) one can show that for every $\mu\in S$ there exists
a unique positive NAF $A$ of $\mathbb{X}'$ in the Revuz
correspondence with $\mu$. In what follows we will denote it by
$A^{\mu}$. For $\mu\in\MM_0(D_T)$ we put
$A^\mu=A^{\mu^+}-A^{\mu^-}$. Let us also note that if $\mu$
belongs to the subset of $S$ consisting of smooth measures with
respect to the capacity determined by the coercive part of the
time dependent form associated with $\BX'$ (see \cite{RT} for
details), then by \cite[Theorem 2.2]{RT} the AF $A^{\mu}$ is
continuous. In fact, in this  case an explicit representation of
$A^{\mu}$ is known (see Theorem 2.1 and Corollary 4.4 in
\cite{K:JTP}).

Let
\[
\zeta = \inf\{t\ge 0,\,\mathbf{X}_{t}\notin\BR_+\times D\},\quad
\zeta_{\tau}=\zeta\wedge (T-\tau(0))
\]
and let $p'_D$ denote the transition density of the process $\BX'$
killed at first exit from $\BR_+\times D$, i.e.
\[
p'_D(t,(s,x),z)=p'(t,(s,x),z)
-E'_{s,x}[\zeta<t,p'(t-\zeta,\mathbf{X}_{\zeta},z)],
\]
where $p'$ is the transition density of $\BX'$. It is known that
$A^{\mu}$ corresponds to $\mu$ iff for quasi-every $(s,x)\in D_T$,
\begin{equation}
\label{meyer}
E'_{s,x}\int_{0}^{\zeta_{\tau}}f(\mathbf{X}_t)\,dA^{\mu}_{t}
=\int\!\!\int_{[0,T-s]\times D}f(z)p'_{D}(t,(s,x),z)\,d\mu(z)\,dt
\end{equation}
for every $f\in\BB^+(D_T)$ (see \cite{Meyer}).

In \cite{Kl:SPA} a correspondence between two parameter additive
functionals of $\BX$ and soft measures on $D_{T}$ is considered.
An AF $\{A^{\mu}_{s,t}, 0\le s\le t\le T\}$ of $\BX$ corresponds
to a bounded soft measure $\mu$ on $D_T$ in the sense of
\cite{Kl:SPA} if for q.e. $(s,x)\in D_T$,
\begin{equation}
\label{eq2.06} E_{s,x}\int_{s}^{\xi^s\wedge T}
f(t,X_{t})\,dA^{\mu}_{s,t} =\!\int_{[s,T]\times D}
f(t,y)p_{D}(s,x,t,y)\,d\mu(t,y)
\end{equation}
for every $f\in\BB^+(D_T)$, where $\xi^s$ is defined by
(\ref{eq1.09}) and $p_{D}$ is the transition density of the
process $\BX$ killed at the first exit from $D$, i.e.
\[
p_D(s,x,t,y)=p(s,x,t,y)-E_{s,x}[\xi^s<t,p(t-\xi^s,X_{\xi^s},y)].
\]
Let us define $A=\{A_t,t\ge 0\}$ as
\begin{equation}
\label{eq2.009} A_t(\omega')=A^{\mu}_{s,s+t}(\omega),\quad
\omega'=(s,\omega),\,t\ge0.
\end{equation}
Using (\ref{eq2.07})--(\ref{eq2.08}) and the fact that
$\zeta_{\tau}(\omega')=(\xi^s(\omega)-s)\wedge(T-s)$ one can show
that (\ref{eq2.06}) is satisfied iff (\ref{meyer}) is satisfied
with $A^{\mu}$ replaced by $A$. Therefore a two parameter AF
$\{A^{\mu}_{s,t}\}$ of $\BX$ corresponds to $\mu$ in the sense of
\cite{Kl:SPA} iff the one-parameter AF $A$ of $\BX'$ defined by
(\ref{eq2.009})  corresponds to $\mu$ in the sense of
(\ref{Revuz}).

\begin{lm}
\label{lm.qe} Assume that $\mu\in \MM^{+}_{0,b}(D_{T})$. Then for
q.e. $(s,x)\in D_T$,
\[
E'_{s,x}\int_{0}^{\zeta_{\tau}} dA^{\mu}_{t}<+\infty.
\]
\end{lm}
\begin{dow}
Follows from \cite[Theorem 6.2]{Kl:SPA}.
\end{dow}
\medskip

We say that a measurable function $f:D_T\rightarrow\BR$ is
quasi-integrable if the function $\BR_+\ni t\mapsto
f(\mathbf{X}_{t})$ belongs to $L^1_{loc}(\BR_+)$ $P_{s,x}$-a.s.
for q.e. $(s,x)\in D_T$, where we put $f(s,x)=0$ for $(s,x)\in
D_{T}^{c}$. The set of all quasi-integrable functions on $D_T$
will be denoted by $qL^1(D_T)$.

\begin{wn}
\label{wn.ql} If $f\in L^{1}(D_{T})$ then $f\in qL^{1}(D_{T})$.
\end{wn}

\begin{uw}
Analysis similar to that in \cite[Remark 4.4]{KR} shows that if
for every $\varepsilon>0$ there exists an open set
$G_{\varepsilon}\subset D_{T}$ such that
$\mbox{cap}(G_{\varepsilon})<\varepsilon$ and $f_{|D_{T}\setminus
G_{\varepsilon}}\in L^{1}(D_{T}\setminus G_{\varepsilon})$ then
$f\in qL^{1}(D_{T})$. Since $D_{T}$ is bounded, the reverse
implication is also true. Indeed, assume that $f\in
qL^{1}(D_{T})$, $f\ge0$. Then extending $f$ on
$\mathbb{R}_{+}\times\mathbb{R}^{d}$ by putting zero outside
$D_{T}$ we get that $A_{t}=\int_{0}^{t}f(\mathbf{X}_{r})\,dr$ is
finite for every $t\ge 0$, which implies that it is a positive
continuous AF of $\BX'$. Directly from the definition of the Revuz
correspondence it follows that $f\cdot m_{1}$ corresponds to the
AF $A$. This implies in particular that $f\cdot m_{1}$ is a smooth
measure. Therefore for every compact $K\subset
\mathbb{R}_{+}\times\mathbb{R}^{d}$ and for every $\varepsilon >0$
there exists an open set $U_{\varepsilon}\subset
\mathbb{R}_{+}\times\mathbb{R}^{d}$ such that
cap$(U_{\varepsilon})\le \varepsilon$ and $f\in L^{1}(K\setminus
U_{\varepsilon})$. Taking $K=\overline{D_{T}}$ we get the desired
result.
\end{uw}

\begin{lm}
\label{lm2.3} Let $\mu$ be a positive smooth measure on $D_T$ and
let $\nu\in\MM_{0,b}^{+}(D_{T})$. If
\[
E'_{s,x} \int_{0}^{\zeta_{\tau}}dA_{r}^{\mu}\le E'_{s,x}
\int_{0}^{\zeta_{\tau}}dA_{r}^{\nu}
\]
for q.e. $(s,x) \in D_{T}$ then $\|\mu\|_{TV} \le \|\nu\|_{TV}$.
\end{lm}
\begin{dow}
By the assumptions, for q.e. $(s,x)\in D_{T}$,
\begin{equation}
\label{eq.fun} \int_0^T\!\!\int_{D}p_{D}(s,x,r,y)\,d\mu(r,y)
\le \int_0^T\!\!\int_{D}p_{D}(s,x,r,y)\,d\nu(r,y).
\end{equation}
Let $D^{k}=\{x\in D;\mbox{dist}(x,\partial D)>1/k\}$,
$D^{k}_{T}=[1/k,T]\times D^{k}$ and
$h_{k}=\mathbf{1}_{\{D^{k}_{T}\}}$. By \cite{MP}, for every $k$
there exists a minimal solution $e_{k}\in L^{2}(0,T;H^{1}(D))$, in
the variational sense, of the obstacle problem
\[
\left\{
\begin{array}{l}\frac{\partial e_{k}}{\partial t}-A_{t}e_{k}\ge
0,\quad \frac{\partial e_{k}}{\partial t}-A_{t}e_{k}=0
\mbox{ on }\{e_{k}>h_{k}\}, \smallskip\\
e_{k}\ge h_{k}, \smallskip\\
e_{k}(0, \cdot)=0,\quad e_{k}(t,\cdot)_{|\partial D}=0,\quad
t\in(0,T).
\end{array}
\right.
\]
In particular $\beta_{k}=\frac{\partial e_{k}}{\partial t}
-A_{t}e_{k}$ belongs to $\MM^{+}_{b}(D_{T})\cap \WW'(D_{T})$,
where $\WW'(D_T)$ is the dual space of $\WW(D_T)$. In fact,
$\beta_k\in\MM^{+}_{0,b}(D_{T})\cap \WW'(D_{T})$ because it is
known that $\MM^{+}_{b}(D_{T})\cap
\WW'(D_{T})=\MM^{+}_{0,b}(D_{T})\cap \WW'(D_{T})$ (see, e.g.,
\cite{DPP}). By \cite{Aronson}, for q.e. $(s,x)\in D_{T}$,
\begin{equation}
\label{eq.aron} e_{k}(s,x)=\int_0^T\!\!\int_{D}
p_{D}(r,y,s,x)\,d\beta_{k}(r,y).
\end{equation}
From (\ref{eq.fun}), (\ref{eq.aron}) and the fact that $0\le e_{k}
\le1$ q.e. on $D_{T}$ it follows that
\begin{equation}
\label{eq2.011} \int_{D_{T}}e_{k}(s,x)\,d\mu(s,x) \le
\int_{D_{T}}e_{k}(s,x)\,d\nu(s,x) \le \|\nu\|_{TV}.
\end{equation}
The fact that $0\le e_{k}\le 1$ q.e. follows from the fact that
$e_{k}$ is the minimal potential majorizing $h_{k}$ (see
\cite[Theorem 3.2.28]{BL}). Indeed, it is an elementary check that
$e_{k}\wedge 1$ is a potential majorizing $h_{k}$. Therefore
$e_{k}\le e_{k}\wedge 1$, which implies that $e_{k}\le 1$. Of
course $e_{k}\ge 0$ since $e_{k}\ge h_{k}$. From this we get in
particular  that $e_{k}(s,x)=1$ for $(s,x)\in D^{k}_{T}$.
Therefore $e_{k}(s,x)\nearrow1$ q.e. on $D_{T}$. Consequently,
letting $k\rightarrow\infty$ in (\ref{eq2.011}) and using Fatou's
lemma we get the desired result.
\end{dow}

\nsubsection{Backward stochastic differential equations}
\label{sec3}

Results of this section together with some results on BSDEs proved
in \cite{Kl:BSM} form the basis for our main theorems on the
obstacle problem (\ref{eq1.1}).

\subsubsection{General BSDEs}
\label{sec3.1}

Suppose we are given a filtered probability space $(\Omega,
\mathcal{G}, \{\mathcal{G}_{t}\}_{t\ge0},P)$. We will need the
following spaces of processes.

$\VV$ (resp. $\VV^+$) is the space of all progressively measurable
c\`adl\`ag processes $V$ of finite variation (resp. increasing
processes) such that $V_0=0$. $\VV^1$ is the space of all
processes $V\in\VV$ such that $E|V|_T<+\infty$, where $|V|_T$
denotes the total variation of $V$ on $[0,T]$.

$\DM^q$, $q>0$, is the space all progressively measurable
c\`adla\`g processes $\eta$ such that $E\sup_{0\le t\le T}
|\eta_t|^q<+\infty$ for every $T>0$.

$M^q$, $q>0$, is the space of all progressively measurable
processes $\eta$ such that
$E(\int^T_0|\eta_t|^2\,dt)^{q/2}<+\infty$ for every $T>0$.

Let us also recall that a c\`adl\`ag adapted process $\eta$ is
said to be of Doob's class (D) if the collection
$\{\eta_{\tau}:\tau$ is a finite $\{\GG_t\}_{t\ge0}$-stopping
time\} is uniformly integrable.

Let $B$  be a  $\{\mathcal{G}_{t}\}$-Brownian motion, $\sigma$ be
a bounded stopping time, $\xi$ be a
$\mathcal{G}_{\sigma}$-measurable random variable, $A\in\VV$ and
let  $f:\Omega \times \mathbb{R}_{+}\times\mathbb{R}\times
\mathbb{R}^{d}\rightarrow\mathbb{R}$ be a function such that
$f(\cdot,y,z)$ is progressively measurable for $y\in\BR,z\in\BR^d$
(for brevity, in our notation we omit the dependence of $f$ on
$\omega\in\Omega$). Let us recall that a pair $(Y,Z)$ of
$\mathbb{R}\times\mathbb{R}^{d}$-valued processes is called a
solution of BSDE$(\xi,\sigma,f+dA)$ if
\begin{enumerate}
\item[(a)]$Y, Z$ are $\{\mathcal{G}_{t}\}$-progressively measurable,
$Y$ is c\`adl\`ag, $P(\int^{\sigma}_0|Z_t|^2\,dt<+\infty)=1$,
\item[(b)]$t\mapsto f(t,Y_{t},Z_{t})\in L^{1}(0,\sigma)$,
$P$-a.s.,
\item[(c)]$Y_{t}=\xi+\int_{t}^{\sigma} f(s, Y_{s},Z_{s})\,ds
+\int_{t}^{\sigma}dA_{s}-\int_{t}^{\sigma}Z_{s}\, dB_{s}$, $0 \le
t\le\sigma$, $P$-a.s.
\end{enumerate}

\begin{uw}
\label{rem3.1} The processes $Y,Z$ in the above definition may be
considered as processes on $\mathbb{R}_{+}$. Indeed, if we set
$Z_{t}=0$ for $t\ge\sigma$ and $Y_{t}=Y_{\sigma}$ for $t\ge
\sigma$ then the equation in condition (c) is satisfied for every
$t\ge 0$ if we adopt the convention that $\int_{a}^{b}=0$ for
$a\ge b$.
\end{uw}

Let us consider the following assumptions.
\begin{enumerate}
\item [(A1)] $f(\cdot,y,z)$ is $\{\GG_t\}$-progressively measurable
for every $(y,z)\in\mathbb{R}\times\mathbb{R}^{d}$ and
$f(t,\cdot,z)$ is continuous for every $(t,z)\in\BR\times\BR^{d}$,
$P$-a.s.,
\item[(A2)]There is $L>0$ such that
$|f(t,y,z)- f(t,y,z')|\le L|z-z'|$ for a.e. $t\ge0$ and every
$y\in\BR$, $z,z'\in\mathbb{R}^{d}$, $P$-a.s.,
\item[(A3)]There is $\kappa\in\BR$ such that
$(f(t,y,z)- f(t,y',z))(y-y')\le\kappa|y-y'|^{2}$ for a.e. $t\ge0$
and every $y,y'\in\BR$, $z\in\mathbb{R}^{d}$, $P$-a.s.,
\item[(A4)]$E|\xi|+E\int_{0}^{\sigma}|f(t,0,0)|\,dt
+E\int_{0}^{\sigma}d|A|_{t}<+\infty$,
\item[(A5)]For every $r>0$, $t\mapsto\sup_{|y|\le r}|f(t,y,0)|
\in L^{1}(0,\sigma)$, $P$-a.s.,
\end{enumerate}
\begin{enumerate}
\item[(AZ)]There exist $\alpha\in(0,1)$, $\gamma\ge0$ and a nonnegative
progressively measurable process $g$ such that $E\int_{0}^{\sigma}
|g_{t}|\,dt<+\infty$ and $P$-a.s.,
\[
|f(t,y,z)- f(t,y,0)|\le\gamma (g_{t}+|y|+|z|)^{\alpha}
\]
for a.e. $t\ge0$ and every $y\in\BR$, $z\in\BR^{d}$.
\end{enumerate}

It is known (see \cite[Theorem 3.11]{Kl:BSM}) that under
(A1)--(A5), (AZ) there exists a unique solution $(Y,Z)$ of
BSDE$(\xi, \sigma, f+dA)$ such that $Y$ is of class (D), $Y\in
\mathcal{D}^{q}$  and $Z\in M^{q}$ for any $q\in(0,1)$.

\subsubsection{Markov type BSDEs}

Let $T>0$, $f:D_T\times \BR\times \mathbb{R}^{d}\rightarrow
\mathbb{R}$ be a measurable function, $\varphi\in L^{1}(D)$. By
putting $\varphi(t,x)=\varphi(x)$ for $(t,x)\in D_T$ we will
regard $\varphi$ as defined on $D_T$. Let $\mu$ be a soft measure
on $D_{T}$ and $B$ be the Brownian motion defined by
(\ref{eq1.8}).

Let $(s,x)\in D_T$. We say that a pair $(Y^{s, x},Z^{s, x})$
consisting of an $\BR$-valued process $Y^{s,x}$ and an
$\BR^d$-valued process $Z^{s,x}$ is a solution of
BSDE$_{s,x}$$(\varphi,D,f+d\mu)$ if $Y^{s,x},Z^{s,x}$ are
$\{\GG'_{t}\}$-progressively measurable, $Y^{s,x}$ is c\`adl\`ag,
$t\mapsto f(\mathbf{X}_{t}, Y^{s,x}_{t},Z_{t}^{s,x})\in
L^{1}(0,\zeta_{\tau})$, $P'_{s,x}$-a.s.,
$P'_{s,x}(\int_{0}^{\zeta_{\tau}}|Z_{r}^{s,x}|^{2}\,dr<+\infty)=1$
and
\begin{align}
\label{eq2.1} \nonumber Y_{t}^{s,x}&=  \mathbf{1}_{\{\zeta > T-
\tau(0)\}} \varphi(\mathbf{X}_{T-\tau(0)})
+\int_{t}^{\zeta_{\tau}}f(\mathbf{X}_{r},Y_{r}^{s,x},
Z^{s,x}_{r})\,dr \\
&\quad+ \int_{t}^{\zeta_{\tau}}dA_{r}^{\mu}
-\int_{t}^{\zeta_{\tau}} Z_{r}^{s,x}\,dB_{r},\quad
t\in[0,\zeta_{\tau}]
\end{align}
holds true $P'_{s,x}$-a.s. In other words, $(Y^{s,x},Z^{s,x})$ is
a solution of (\ref{eq2.1}) if it is a solution of
BSDE$(\mathbf{1}_{\{\zeta>T-\tau(0)\}}\,
\varphi(\mathbf{X}_{T-\tau(0)}),\zeta_{\tau},
f(\mathbf{X},\cdot,\cdot)\,+dA^{\mu})$ on the filtered probability
space $(\Omega', \GG'_{\infty},\{\GG'_t\},P'_{s,x})$ with the
Brownian motion $B$.

\begin{stw}
\label{stw2.1} Let $(s,x)\in D_{T}$ and let
$(Y^{s,x},Z^{s,x})=(Y,Z)$ be a solution of
\mbox{\rm(\ref{eq2.1})}. Assume that \mbox{\rm(A1), (A2), (AZ)}
are satisfied, $Z\in M^{q}$ for some $q>\alpha$, there exist
$0<h\le T-s$ and $\Lambda\in\GG'_{\infty}$ such that
$P'_{s,x}(\Lambda)=1$, $P'_{s,x}(\theta_{h}^{-1}(\Lambda))=1$ and
\mbox{\rm(\ref{eq2.1})} is satisfied for every $\omega\in\Lambda$.
Then
\[
Y_{t-h}\circ \theta_{h}=Y_{t},\quad  t\ge h, \quad P'_{s,x}
\mbox{-a.s.}
\]
and
\[
Z_{\cdot-h} \circ \theta_{h}=Z, \quad dt\otimes
P'_{s,x}\mbox{-a.s.}\mbox{ on } [h,\zeta_{\tau}]\times\Omega'.
\]
\end{stw}
\begin{dow}
Put $\kappa_{h}=(\zeta \circ \theta_{h})\wedge(T-\tau(0)
\circ\theta_{h})$. Then
\begin{align*}
Y_{t}\circ\theta_{h}&= \mathbf{1}_{\{\zeta \circ\theta_{h}>T
-\tau(0)\circ\theta_{h}\}}\varphi(\mathbf{X}_{T-\tau(0) \circ
\theta_{h}}) +\int_{t}^{\kappa_{h}}f(\mathbf{X}_{r} \circ
\theta_{h},Y_{r}\circ\theta_{h}, Z_{r}\circ\theta_{h})\,dr \\
&\quad + \int_{t}^{\kappa_{h}}d(A_r^{\mu}\circ\theta_h) -
\left(\int_{t}^{\zeta_{\tau}}Z_{r}\,dB_{r}\right)\circ \theta_{h},
\quad t \ge 0, \quad P'_{s,x}\mbox{-a.s.}
\end{align*}
Fix $t\ge h$. Since $\zeta\circ\theta_{h}=(\zeta-h)^{+}$ and
$\tau(0)\circ\theta_{h}=\tau(0)+h$, we have $\kappa_h=(\zeta
-h)\wedge (T-\tau(0)-h)$ on the set $\{\zeta\ge h\}$.
Consequently, $h+\kappa_h=\zeta_{\tau}$ on the set $\{\zeta\ge
h\}$. Therefore, since
$\mathbf{X}_{r}\circ\theta_{h}=\mathbf{X}_{r+h}$ and $A^{\mu}$ is
additive, we have
\[
\int_{t}^{\kappa_{h}}f(\mathbf{X}_{r}\circ
\theta_{h},Y_{r}\circ\theta_{h}, Z_{r}\circ\theta_{h})\,dr
=\int_{t+h}^{\zeta_{\tau}}f(\mathbf{X}_{r}, Y_{r-h}\circ
\theta_{h}, Z_{r-h}\circ\theta_{h})\,dr
\]
and
\[
\int_{t}^{\kappa_{h}}d(A_{r}^{\mu} \circ \theta_{h})
=\int_{t+h}^{\zeta_{\tau}}dA_{r}^{\mu}
\]
$P'_{s,x}$-a.s. on $\{\zeta\ge h\}$. Let
$t=t_{0}<t_{1}<\ldots<t_{n}=\zeta_{\tau}$ and let $t'_i=t_{i}+h$,
$i=0,\dots,n$. Since $B$ is additive,
$(Z_{t_i}(B_{t_{i+1}}-B_{t_{i}}))\circ\theta_{h}
=Z_{t'_i-h}\circ\theta_h(B_{t'_{i+1}}-B_{t'_i})$, from which it
follows that
\[
(\int_{t}^{\zeta_{\tau}}Z_{r} \, dB_{r})\circ \theta_{h}
=\int_{t+h}^{\zeta_{\tau}} Z_{r-h}\circ\theta_{h}\,dB_{r}
\]
$P'_{s,x}$-a.s. on $\{\zeta\ge h\}$.  Thus
\begin{align*}
Y_{t} \circ \theta_{h}&
=\mathbf{1}_{\{\zeta>T-\tau(0)\}}\varphi(\mathbf{X}_{T-\tau(0)})
+\int_{t+h}^{\zeta_{\tau}}f(\mathbf{X}_{r}, Y_{r-h}\circ
\theta_{h}, Z_{r-h}\circ \theta_{h})\,dr \\
&\quad+\int_{t+h}^{\zeta_{\tau}}dA_{r}^{\mu}
-\int_{t+h}^{\zeta_{\tau}}Z_{r-h}\circ \theta_{h}\,dB_{r},
\end{align*}
and hence
\begin{align*}
Y_{t-h}\circ\theta_{h}&
=\mathbf{1}_{\{\zeta>T-\tau(0)\}}\varphi(\mathbf{X}_{T-\tau(0)})
+\int_{t}^{\zeta_{\tau}}f(\mathbf{X}_{r}, Y_{r-h}\circ \theta_{h},
Z_{r-h}\circ \theta_{h})\,dr \\
&\quad+\int_{t}^{\zeta_{\tau}}dA_{r}^{\mu}
-\int_{t}^{\zeta_{\tau}}Z_{r-h}\circ \theta_{h} \,dB_{r}
\end{align*}
$P'_{s,x}$-a.s. on $\{\zeta\ge h\}$. On the set $\{\zeta<h\}$ we
have $Y_{t-h}\circ\theta_{h}=0=Y_{t}$ for $t\ge h$, because by the
assumption, $h\le T-s$. By what has been proved the pair
$(Y_{t-h}\circ\theta_{h}, Z_{t-h}\circ\theta_{h})$ solves BSDE
(\ref{eq2.1}) on $[h,\zeta_{\tau}]$, so the desired result follows
from uniqueness of solutions of (\ref{eq2.1}).
\end{dow}
\medskip

In the sequel we say that a Borel set $N\subset D_T$ is properly
exceptional if $m_1(N)=0$ and $P'_{s,x}(\exists\,t\ge0:X_t\in
N)=0$ for every $(s,x)\in N^c$.
\begin{lm}
\label{lm2.1} Let $\Lambda\in\GG'_{\infty}$. If
$P'_{s,x}(\Lambda)=1$ for q.e. $(s,x)\in D_{T}$ then there exists
a properly exceptional set $N\subset D_{T}$ such that
$P'_{s,x}(\theta_{h}^{-1}(\Lambda))=1$ for every $h>0$ and
$(s,x)\in N^{c}$.
\end{lm}
\begin{dow}
By the assumption, $P'_{s,x}(\Lambda^{c})=0$ for q.e. $(s,x) \in
D_{T}$. It is known (see \cite[Corollary 1.8.6]{BB} and
\cite[(6.12)]{GS}) that there exists a properly exceptional set $N
\subset D_{T}$ such that $P'_{s,x}(\Lambda^{c})=0$ for every
$(s,x) \in N^{c}$. Let $(s,x)\in N^{c}$. Then by the Markov
property, $P'_{s,x}(\theta_{h}^{-1}\Lambda^{c})
=E'_{s,x}(P_{\mathbf{X}_{h}}(\Lambda^{c}))$. By the assumption and
the definition of $N$,
$E'_{s,x}P'_{\mathbf{X}_{h}}(\Lambda^{c})=0$ for every $(s,x) \in
N^{c}$, which completes the proof.
\end{dow}
\medskip

In the case of Markov type equations the following hypotheses are
analogues of (A1)--(A5) and (AZ).
\begin{enumerate}
\item[(H1)]$f(\cdot,\cdot,y,z)$
is measurable for every $y\in\BR$, $z\in\BR^d$ and
$f(t,x,\cdot,\cdot)$ is continuous for a.e. $(t,x)\in D_{T}$,
\item[(H2)]There is $L>0$ such that
$|f(t,x,y,z)- f(t,x,y,z')|\le L|z-z'|$ for every $y\in\mathbb{R}$,
$z,z'\in\mathbb{R}^{d}$ and a.e. $(t,x)\in D_{T}$,
\item[(H3)]There is $\kappa\in\BR$ such that
$(f(t,x,y,z)-f(t,x,y',z))(y-y')\le\kappa|y-y'|^{2}$ for every
$y,y'\in\BR$, $z\in\BR^d$ and a.e. $(t,x)\in D_{T}$,
\item[(H4)]$f(\cdot,\cdot,0,0) \in L^{1}(D_{T})$, $\mu\in\MM_{0,b}(D_{T})$,
$\varphi\in L^{1}(D)$,
\item[(H5)]$\forall_{r>0}\,\,(t,x)\mapsto\sup_{|y|\le r}|f(t,x,y,0)|
\in qL^{1}(D_{T})$.
\item[(HZ)]There exists $\alpha \in (0,1),\gamma \ge 0$ and
$\varrho \in L^{1}(D_{T}),\varrho \ge 0$ such that
\[
|f(t, x, y, z)- f(t, x, y, 0)| \le \gamma(\varrho(t, x)
+|y|+|z|)^{\alpha}
\]
for every $y\in\BR$, $z\in\BR^d$ and a.e. $(t,x)\in D_T$.
\end{enumerate}

\begin{lm}
\label{lm3.1} Assume that \mbox{\rm(H1)--(H5)} are satisfied. Then
for q.e. $(s,x)\in D_{T}$ the data $\varphi(\mathbf{X}_{T-
\tau(0)}) \mathbf{1}_{\{\zeta>T-\tau(0)\}}$, $\zeta_{\tau}$, $f(\bfX, \cdot,
\cdot)$, $A^{\mu}$ satisfy \mbox{\rm(A1)--(A5)} under the measure
$P'_{s,x}$.
\end{lm}
\begin{dow}
Follows from Lemma \ref{lm.qe}.
\end{dow}

\begin{stw}
\label{stw3.5} Assume \mbox{\rm(H1)--(H5)} and \mbox{\rm(HZ)}.
Then for q.e. $(s,x) \in D_{T}$ there exists a solution $(Y^{s,x},
Z^{s,x})$ of \mbox{\rm(\ref{eq2.1})} and there exists a pair of
processes $(Y,Z)$ such that for q.e. $(s,x)\in D_{T}$,
\begin{equation}
\label{eq3.9} Y^{s,x}_{t}=Y_{t},\quad t\ge 0,\,P'_{s,x}
\mbox{-a.s.}, \quad Z^{s,x}=Z, \quad dt\otimes P'_{s,x}\mbox{-a.e.
on }[0,\zeta_{\tau}]\times\Omega'.
\end{equation}
Moreover, there exist  Borel measurable functions
$u,\psi:D_T\rightarrow\BR$ such that for q.e. $(s,x)\in D_{T}$,
\[
Y_{t}=u(\mathbf{X}_{t}), \quad P'_{s,x}\mbox{\rm-a.s.}
\]
for every $t\in[0,T-\tau(0)]$ and
\[
\psi(\mathbf{X})=Z,\quad dt\otimes P'_{s,x}\mbox{-a.e. on }
[0,\zeta_{\tau}]\times\Omega'.
\]
If $f$ does not depend on $z$ then there is $C$ depending only on
$\kappa,T$ such that
\begin{align}
\label{eq3.03} E'_{s,x}\int^{\zeta_{\tau}}_0|f(\mathbf{X}_{r},
u(\mathbf{X}_{r}))|\,dr &\le C E'_{s,x}
\Big(\mathbf{1}_{\{\zeta>T-\tau(0)\}}
|\varphi (\mathbf{X}_{T-\tau(0)})|\nonumber\\
&\qquad+\int^{\zeta_{\tau}}_0|f(\mathbf{X}_{r},0)|\,dr
+\int^{\zeta_{\tau}}_0d|A^{\mu}|_r\Big)
\end{align}
for q.e. $(s,x)\in D_T$.
\end{stw}
\begin{dow}
By Lemma \ref{lm.qe}, $E'_{s,x}|A^{\mu}|_{\zeta_{\tau}}<+\infty$
for q.e. $(s,x) \in D_{T}$. Moreover, for $(s,x)\in D_{T}$,
\begin{align*}
E'_{s,x} \mathbf{1}_{\{\zeta> T-\tau(0)\}}
|\varphi(\mathbf{X}_{T-\tau(0)})| &\le E'_{s,x}
|\varphi(\mathbf{X}_{T-\tau(0)})| \\
&\le\int_{D}p(s,x,T,y)|\varphi(y)|\,dy \le c
\int_{D}|\varphi(y)|\,dy.
\end{align*}
Therefore by \cite[Proposition 3.10]{Kl:BSM}, for q.e. $(s,x)\in
D_{T}$ there exists a solution
$(\overline{Y}^{s,x},\overline{Z}^{s,x})$ of the equation
\[
\overline{Y}_{t}^{s,x}= \mathbf{1}_{\{\zeta > T-\tau(0)\}}\varphi
(\mathbf{X}_{T-\tau(0)}) + \int_{t}^{\zeta_{\tau}}
f(\mathbf{X}_{r}, 0, 0)\, dr + \int_{t}^{\zeta_{\tau}}
dA_{r}^{\mu} - \int_{t}^{\zeta_{\tau}} \overline{Z}^{s,x}_{r}\,
dB_{r},\quad t\ge0
\]
such that $\overline{Y}^{s,x}$ is of class (D) and
$(\overline{Y}^{s,x},\overline{Z}^{s,x})\in \mathcal{D}^{q}\otimes
M^{q}$ for any $q\in(0,1)$. Hence
\[
\overline{Y}_{t}^{s,x}= E'_{s,x}\Big(\mathbf{1}_{\{\zeta >
T-\tau(0)\}} \varphi (\mathbf{X}_{T-\tau(0)})
+\int_{t}^{\zeta_{\tau}} f(\mathbf{X}_{r}, 0, 0)\,dr
+\int_{t}^{\zeta_{\tau}} dA_{r}^{\mu} |\GG'_{t}\Big),
\]
so from \cite[Lemma A.3.5]{Fukushima} it follows that there exists a
process $\overline{Y}$ such that
$\overline{Y}^{s,x}_{t}=\overline{Y}_{t}$, $t\ge 0$,
$P'_{s,x}$-a.s. for q.e. $(s,x)\in D_{T}$. Since
\[
\overline{Z}^{s,x}_{t}=E'_{s,x}\Big(\mathbf{1}_{\{\zeta >
T-\tau(0)\}}\varphi (\mathbf{X}_{T-\tau(0)}) +
\int_{0}^{\zeta_{\tau}} f(\mathbf{X}_{r}, 0, 0)\, dr
+\int_{0}^{\zeta_{\tau}} dA_{r}^{\mu}|\GG'_{t}\Big)
-\overline{Y}_{0},
\]
using once again \cite[Lemma A.3.5]{Fukushima} we conclude that there
exists a process $\overline{Z}$ such that
$\overline{Z}^{s,x}=\overline{Z}$, $dt\otimes P'_{s,x}$-a.e. on
$[0,\zeta_{\tau}]\times\Omega'$ for q.e. $(s,x)\in D_{T}$. Thus,
for q.e. $(s,x)\in D_T$,
\[
\overline{Y}_{t}= \mathbf{1}_{\{\zeta>T-\tau(0)\}}\varphi
(\mathbf{X}_{T-\tau(0)}) + \int_{t}^{\zeta_{\tau}}
f(\mathbf{X}_{r}, 0, 0)\,dr+\int_{t}^{\zeta_{\tau}} dA_{r}^{\mu}
-\int_{t}^{\zeta_{\tau}}\overline{Z}_{r}\, dB_{r},\quad t\ge0
\]
$P'_{s,x}$-a.s. From this, the method of construction of the
solution of BSDE$_{s,x}(\varphi,D,f+d\mu)$ (see the proof of
\cite[Proposition 3.10]{Kl:BSM}) and repeated application of
\cite[Lemma A.3.3]{Fukushima} we deduce the existence of a pair
$(Y,Z)$ such that (\ref{eq3.9}) is satisfied. Let
$\Lambda\subset\Omega'$ be a set of those $\omega\in\Omega'$ for
which $Y$ is c\`adl\`ag and
\[
Y_{t}=\mathbf{1}_{\{\zeta> T-\tau(0)\}}
\varphi(\mathbf{X}_{T-\tau(0)})
+\int_{t}^{\zeta_{\tau}}f(\mathbf{X}_{r}, Y_{r}, Z_{r})\,dr
+\int_{t}^{\zeta_{\tau}}dA^{\mu}_{r}
-\int_{t}^{\zeta_{\tau}}Z_{r}\,dB_{r}, \quad t\ge0.
\]
Then $\Lambda\in\GG'_{\infty}$ and of course $P'_{s,x}(\Lambda)=1$
for q.e. $(s,x) \in D_{T}$. From this and Proposition \ref{stw2.1}
and Lemma \ref{lm2.1} it follows that if we set
$u(s,x)=E'_{s,x}Y_0$ for $(s,x)\in D_T$ then for every $t \in
[0,T-s]$,
\[
u(\mathbf{X}_{t})=E'_{\mathbf{X}_{t}}Y_{0}
=E'_{s,x}(Y_{0}\circ\theta_{t}|\GG'_{t})
=E'_{s,x}(Y_{t}|\GG'_{t})=Y_{t}, \quad P'_{s,x}\mbox{-a.s.}
\]
for q.e. $(s,x)\in D_{T}$. Put $C_{t}=\int_{0}^{t}Z_{r}\,dr$,
$t\ge0$. Since $Z\mathbf{1}_{[0,\zeta_{\tau}]}=Z$, using the
property of $Z$ proved in Proposition \ref{stw2.1} one can check
that for every $h\ge 0$,
\[
C_{(t+h)\wedge(T-\tau(0))}=C_{t\wedge(T-\tau(0))}
+C_{h\wedge(T-\tau(0))}\circ\theta_{t\wedge(T-\tau(0))}.
\]
Accordingly, $C$ is a continuous AF of the process $\mathbf{X}$
killed at first exit from $D_{T}$. Therefore it follows from
\cite[Theorem 66.2]{Sharpe} that there exists a Borel measurable
function $\psi$ on $D_{T}$ such that $Z=\psi(\mathbf{X})$,
$dt\otimes P'_{s,x}$-a.s. on $[0,\zeta_{\tau}]\times\Omega'$ for
q.e. $(s,x)\in D_{T}$. Finally, by Tanaka's formula,
\begin{equation}
\label{eq3.04}
-\int^t_0\hat{Y}_{s}f(r,\mathbf{X}_{r},Y_{r})\,dr\le
|Y_{t}|-|Y_{0}| +\int^t_0 \hat{Y}_{r-}\,dA^{\mu}_{r}-\int^t_0
\hat{Y}_{r}\,dB_{r},\quad 0\le t\le\zeta_{\tau}.
\end{equation}
If $\kappa\le0$ then by (H3),
$-\hat{Y}_{r}(f(r,\mathbf{X}_{r},Y_{r}))\ge
|f(r,\mathbf{X}_{r},Y_{r})|-|f(r,\mathbf{X}_{r},0)|$. On
substituting this into (\ref{eq3.04}) and then taking expectation
on both sides of the inequality we get (\ref{eq3.03}) with $C=1$.
The general case can be reduced to the case $\kappa\le0$ by the
standard change of variables.
\end{dow}
\medskip

In what follows we denote by $T_k$, $k\ge0$, the truncation
operator, i.e.
\[
T_k(y)=(-k)\vee(y\wedge k),\quad y\in\BR.
\]
\begin{lm}
\label{lm2.4} Assume that $\psi_{n},\psi$ are Borel measurable
functions on $D_{T}$ such that $\psi_{n}(\mathbf{X})\rightarrow
\psi(\mathbf{X})$, $dt\otimes P'_{s,x}$-a.e. on
$[0,\zeta_{\tau}]\times\Omega'$ for q.e. $(s,x)\in D_{T}$. Then
for some subsequence (still denoted by $n$), $\psi_{n}\rightarrow
\psi$, $m_1$-a.e. on $D_{T}$.
\end{lm}
\begin{dow}
Let $(s,x)\in D_{T}$ be such that $\psi_{n}(\mathbf{X})\rightarrow
\psi(\mathbf{X})$, $dt\otimes P'_{s,x}$-a.s. on
$[0,\zeta_{\tau}]\times\Omega'$. Then
\[
E'_{s,x}
\int_{0}^{\zeta_{\tau}}T_{k}(|\psi_{n}-\psi|)(\bX_{r})\,dr
\rightarrow0
\]
as $n\rightarrow\infty$. Since
\[
E'_{s,x}\int_{0}^{\zeta_{\tau}}T_{k}(|\psi_{n}-\psi|)(\bX_{r})\,dr
=\int_{D_{T}}p_{D}(s,x,\theta,y)T_{k}(|\psi_{n}-\psi|)(\theta,y)
\,d\theta\,dy,
\]
there exists a subsequence (still denoted by $n$) such that
$p_{D}(s,x,\theta,y)T_{k}(|\psi_{n}-\psi|)\rightarrow 0$,
$m_{1}$-a.e. on $[s,T]\times D$, which by the positivity of
$p_{D}(s,x,\cdot,\cdot)$ and the definition of $T_{k}$ implies
that $\psi_{n}\rightarrow\psi$, $m_{1}$-a.e. on $[s,T]\times D$.
Since $(s,x)$ can be chosen so that $s$ is arbitrary close to
zero, one can choose a further subsequence (still denoted by $n$)
such that $\psi_{n}\rightarrow\psi$, $m_{1}$-a.e. on $D_{T}$.
\end{dow}
\medskip

We say that a measurable function $u:D_T\rightarrow\BR$ is of
class (FD) if for q.e. $(s,x)\in D_T$ the process $[s,T]\ni
t\mapsto u(t,X_t)$ is of Doob's class (D) under the measure
$P_{s,x}$.

We say that  a measurable function $u:D_T\rightarrow\BR$ is
quasi-c\`adl\`ag if for q.e. $(s,x)\in D_T$ the process $[s,T]\ni
t\mapsto u(t,X_t)$ is c\`adl\`ag under $P_{s,x}$.

$\mathcal{FD}$ is the set of all quasi-c\`adl\`ag functions
$u:D_T\rightarrow\BR$. $\mathcal{FS}^{q}$ (resp.
$\mathcal{FD}^{q}$), $q>0$, is the set of all quasi-continuous
(resp. quasi-c\`adl\`ag) functions $u:D_T\rightarrow\BR$ such that
for q.e. $(s,x)\in D_T$,
\[
E_{s,x}\sup_{s\le t\le T}|u(t,X_{t})|^{q}<+\infty.
\]

$FM^{q}$, $q>0$, is the set of all measurable functions
$u:D_T\rightarrow\BR$ such that for q.e. $(s,x)\in D_T$,
\[
E_{s,x}(\int^T_s|u(t,X_t)|^2\,dt)^{q/2}<+\infty.
\]

$\mathcal{T}^{0,1}_{2}$ is the set of all measurable functions $u$
on $D_{T}$ such that for every $k\ge0$, $T_{k}(u)\in
L^{2}(0,T;H^{1}_0(D_{T}))$. From \cite[Lemma 2.1]{BBGGPV} it
follows that for every $u\in\mathcal{T}^{0,1}_{2}$ there exists a
unique measurable function $v$ on $D_{T}$ such that $\nabla
T_{k}(u)=\mathbf{1}_{\{|u|<k\}}v$,\ $m_{1}$-a.e. We shall set
$\nabla u=v$.

$W^{1,q}_0$, $q\ge1$, is the closure of the space
$C^{\infty}_c(D)$ in the Sobolev space $W^{1,q}$ of functions
$u\in L^2(D)$ having weak derivatives in $L^q(D)$. In particular,
$W^{1,2}_0=H^1_0$.

\begin{stw}
\label{stw3.7} Assume that $\mu\in\MM_{0,b}^{+}(D_{T})$,
$\varphi\in L^{1}(D)$, $\varphi \ge 0$. Then $v:D_T\rightarrow\BR$
defined as
\begin{equation}
\label{eq2.2} v(s,x)=E'_{s,x}\Big(\mathbf{1}_{\{\zeta>T-\tau(0)\}}
\varphi(\bX_{T-\tau(0)})
+\int_{0}^{\zeta_{\tau}}dA_{r}^{\mu}\Big),\quad (s,x)\in D_T
\end{equation}
is of class \mbox{\rm(FD)} and belongs to $\mathcal{FD}^{q}$ for
any $q\in(0,1)$.  Moreover, $v\in\mathcal{T}_{2}^{0,1}$, $v\in
L^{q}(0,T; W_{0}^{1,q})$ for $q\in[1,\frac{d+2}{d+1})$, $\nabla
v\in FM^{q}$ for $q\in(0,1)$ and for every $k\ge0$,
\begin{equation}
\label{eq2.3} E'_{s,x}\int_{0}^{\zeta_{\tau}}|\sigma\nabla
T_{k}(v)|^{2}(r, X_{r})\,dr \le 4k E'_{s,x}
\Big(\mathbf{1}_{\{\zeta>T-\tau(0)\}}\varphi(\bX_{T-\tau(0)})
+\int_{0}^{\zeta_{\tau}}dA_{r}^{\mu}\Big)
\end{equation}
for q.e. $(s,x)\in D_{T}$. Finally, for q.e. $(s,x)\in D_{T}$ the
pair $(v(\bX),\sigma\nabla v(\bX))$ is a solution of \mbox{\rm
BSDE}$_{s,x}(\varphi,D,d\mu)$.
\end{stw}
\begin{dow}
By Proposition \ref{stw3.5} there exists a pair of processes
$(Y,Z)$ such that $Y$ is of class (D),
$(Y,Z)\in\mathcal{D}^{q}\otimes M^{q}$ for $q\in(0,1)$ and for
q.e. $(s,x)\in D_{T}$ the pair $(Y,Z)$ is a solution of
BSDE${}_{s,x}(\varphi,D,d\mu)$, i.e.
\[
Y_{t}=\mathbf{1}_{\{\zeta>T-\tau(0)\}}\varphi(\bX_{T-\tau(0)}) +
\int_{t}^{\zeta_{\tau}}dA_{r}^{\mu}
-\int_{t}^{\zeta_{\tau}}Z_{r}\,dB_{r},\quad t\in[0,\zeta_{\tau}].
\]
By Proposition \ref{stw3.5} again, for $n\in\BN$ there is a pair
of processes $(Y^{n},Z^{n})$ such that $Y^{n}$ is of class (D),
$(Y^{n},Z^{n})\in \mathcal{D}^{q}\otimes M^{q}$ for $q\in(0,1)$
and $(Y^{n},Z^{n})$ is a solution of the following BSDE${}_{s,x}$
\begin{equation}
\label{eq3.14}
Y_{t}^{n}=\mathbf{1}_{\{\zeta>T-\tau(0)\}}\varphi(\bX_{T-\tau(0)})
+ \int_{t}^{\zeta_{\tau}}n(Y_{r}^{n}-Y_{r})^{-}\,dr
-\int_{t}^{\zeta_{\tau}}Z_{r}^{n}\,dB_{r},\quad
t\in[0,\zeta_{\tau}]
\end{equation}
for q.e. $(s,x)\in D_{T}$. By \cite[Theorem 5.7]{Kl:BSM}, for q.e.
$(s,x)\in D_{T}$,
\begin{equation}
\label{eq3.15} Y_{t}^{n}\nearrow Y_{t},\quad t\ge 0,\,P'_{s,x}
\mbox{-a.s.},
\end{equation}
\begin{equation}
\label{eq3.16} Z^{n}\rightarrow Z,\quad dt\otimes
P'_{s,x}\mbox{-a.e. on } [0,\zeta_{\tau}]\times\Omega'.
\end{equation}
Let $A_{t}^{n}=\int_{0}^tn(Y_{r}^{n}-Y_{r})^{-}\,dr$. By
Proposition \ref{stw3.5}, $Y_{t}=v(\bX_{t})$, $P'_{s,x}$-a.s. for
$t\in[0,T-\tau(0)]$ and $Z=\psi(\bX)$, $dt\otimes P'_{s,x}$-a.e.
on $[0,\zeta]\times\Omega'$ for some measurable function $\psi$ on
$D_{T}$. Therefore
\[
A_{t}^{n}= \int_{0}^{t} n(Y_{r}^{n}- v(\bX_{r}))^{-}\,dr, \quad
t\in[0,\zeta_{\tau}], \quad P'_{s,x}\mbox{-a.s.}
\]
From this, (\ref{eq3.14}) and  Proposition \ref{stw3.5},
\begin{equation}
\label{eq3.17}
Y_{t}^{n}=v_{n}(\bX_{t}),\quad P'_{s,x}\mbox{-a.s. for }t\in
[0,\zeta_{\tau}]
\end{equation}
and
\begin{equation}
\label{eq3.18} Z^{n}= \psi_{n}(\bX),\quad dt\otimes
P'_{s,x}\mbox{-a.e. on }[0,\zeta_{\tau}]\times\Omega',
\end{equation}
where
\[
v_{n}(s,x)=E'_{s,x}\Big(\mathbf{1}_{\{\zeta>T-\tau(0)\}}
\varphi(\bX_{T-\tau(0)})
+\int_{0}^{\zeta_{\tau}}\phi_{n}(\mathbf{X}_{r})\,dr\Big),\quad
\phi_{n}=n(v_{n}-v)^{-}
\]
and $\psi_{n}$ is a measurable function on $D_{T}$. Now fix $n\in
\mathbb{N}$ and put $(\overline{Y}, \overline{Z})=(Y^{n},Z^{n})$,
$\overline{\phi}=\phi_{n}$. Then for q.e. $(s,x)\in D_{T}$ the
pair $(\overline{Y}, \overline{Z})$ is a solution of BSDE$_{s,x}
(\varphi,D,\bar{\phi})$. Let $\varphi_{m}= \varphi\wedge m$,
$\overline{\phi}_{m}= \overline{\phi}\wedge m$. By Proposition
\ref{stw3.5}, for q.e. $(s,x) \in D_{T}$ there exists a solution
$(\overline{Y}^{m}, \overline{Z}^{m})$ of BSDE$_{s,x}(\varphi_{m},
D, \overline{\phi}_{m})$ such that $\overline{Y}^{m}$ is of class
(D) and $(\overline{Y}^{m},
\overline{Z}^{m})\in\mathcal{D}^{q}\otimes M^{q}$ for $q\in(0,1)$
(it is known that in fact $\overline{Y}^{m}$ is continuous). By
\cite[Theorem 5.7]{Kl:BSM},
\begin{equation}
\label{eq2.6} \overline{Y}_{t}^{m}\nearrow \overline{Y}_{t},\quad
t\ge 0,\,P'_{s,x}\mbox{-a.s.},
\end{equation}
\begin{equation}
\label{eq2.7} \overline{Z}^{m}\rightarrow\overline{Z},\quad
dt\otimes P'_{s,x} \mbox{-a.e. on } [0,\zeta_{\tau}] \times
\Omega'.
\end{equation}

We divide the proof that $v$ has the desired regularity properties
into 3 steps. \smallskip\\
Step 1. We assume that $\varphi\in L^{2}(D)$ and $\mu=g\cdot
m_{1}$ for some $g\in L^{2}(D_{T})$, $g\ge0$. Let $w\in
\mathcal{W}(D_{T})$ be a weak solution of the problem
\[
\left\{
\begin{array}{l}
\frac{\partial w}{\partial t}+A_{t}w=-g\mbox{ in }D_{T},\\
w(T,\cdot)=\varphi, \quad w(t, \cdot)_{|\partial D}=0,\quad t
\in(0,T).
\end{array}
\right.
\]
Let $\tilde{w}: S_{T}:=[0,T]\times\BR^d\rightarrow \mathbb{R}$ be
an extension of $w$ on $S_{T}$ such that
\[
\tilde{w}(t,x)= \left\{
\begin{array}{l}w(t,x), \quad (t,x) \in D_{T}, \\
w(t,x)=0, \quad (t,x) \in S_{T}\setminus D_{T}.
\end{array}
\right.
\]
Then $\tilde{w}\in\mathcal{W}(S_{T})$ and hence $\frac{\partial
\tilde{w}}{\partial t}+A_{t}\tilde{w} \in L^{2}(0,T;
H^{-1}(\mathbb{R}^{d}))$. Let $g_{0}, G \in L^{2}(S_{T})$ be such
that $\frac{\partial \tilde{w}}{\partial t}+
L_{t}\tilde{w}=-g_{0}-\mbox{div}(G)$. By \cite[Theorem
4.3]{K:JTP}, $\tilde{w}$ has a quasi-continuous $m_{1}$ version
(still denoted by $\tilde{w}$) such that $\tilde{w}\in\FF
S^{2},\nabla \tilde{w} \in FM^{2}$ and for q.e. $(s,x)\in S_{T}$,
\begin{align*}
\tdw(t,X_{t})&=\tdw(s,x)- \int_{s}^{t}g_{0}(r, X_{r})\,dr
-\int_{s}^{t}G(r, X_{r})\,d^{*}X_{r}\\
&\quad + \int_{s}^{t}\sigma \nabla\tdw(r, X_{r})\,dB_{s,r},\quad s
\le t \le T, \quad P_{s,x}\mbox{-a.s.}
\end{align*}
Hence, by Lemma \ref{lm1.1},
\begin{align*}
\tdw(\mathbf{X}_{t})&= \mathbf{1}_{\{\zeta>T-\tau(0)\}}
\tdw(\mathbf{X}_{T-\tau(0)})
+\int_{t}^{\zeta_{\tau}}g(\mathbf{X}_{r})\,dr\\
&\quad-\int_{t}^{\zeta_{\tau}}\sigma\nabla\tilde
w(\mathbf{X}_{r})\,dB_{r}, \quad t\in[0,\zeta_{\tau}],\quad
P'_{s,x}\mbox{-a.s.}
\end{align*}
Observe also that
$E'_{s,x}|(\tdw-\varphi)(\mathbf{X}_{T-\tau(0)})|
=\int_{D}p_{D}(s,x,T,y)|\tdw-\varphi|(T,y)\,dy=0$. Thus
\begin{equation}
\label{eq3.07} \tilde w(s,x)
=E'_{s,x}\Big(\mathbf{1}_{\{\zeta>T-\tau(0)\}}
\varphi(\bfX_{T-\tau(0)}) +\int_{0}^{\zeta_{\tau}}
g(\mathbf{X}_{r})\,dr\Big).
\end{equation}
Therefore $v=w$ satisfies all assertions of the proposition except
from (\ref{eq2.3}). To prove (\ref{eq2.3}), let us fix $k>0$ and
$z\in \mathbb{R}$. Using the convention of Remark \ref{rem3.1}, by
Tanaka's formula and taking expectations,
\begin{align*}
|\tdw(s,x)-z|+E'_{s,x}L^{z}_{\zeta_{\tau}}(\tdw(\mathbf{X}))
&=E'_{s,x}|\mathbf{1}_{\{\zeta>T-\tau(0)\}}
\varphi(\bfX_{T-\tau(0)})-z|\\
&\quad +E'_{s,x}\int_{0}^{\zeta_{\tau}}\hat{\mbox{\rm sgn}}
(\tdw(\bfX_{r})-z)g(\bfX_{r})\,dr,
\end{align*}
where $\{L^z_t(\tilde w(\bfX)),t\ge0\}$ denotes the (symmetric)
local time of $\tilde w(\bfX)$ at $z$  and
$\hat{\mbox{sgn}}(x)={\mathbf{1}}_{x\neq0}\frac{x}{|x|}$\,. Hence
\[
E'_{s,x}L^{z}_{\zeta_{\tau}}(\tdw(\bfX)) \le
E'_{s,x}|\mathbf{1}_{\{\zeta>T-\tau(0)\}}\varphi(\bfX_{T-\tau(0)})
-\tdw(s,x)| + E'_{s,x}\int_{0}^{\zeta_{\tau}}g(\bfX_{r})\,dr.
\]
Multiplying the above inequality by the function
$i(z)=\mathbf{1}_{[-k,k]}(z)$, integrating with respect to $z$ and
applying the occupation time formula and Fubini's theorem, we get
\[
E'_{s,x}\int_{0}^{\zeta_{\tau}}\!|\sigma\nabla
T_{k}(\tdw)|^{2}(\bfX_{r})\,dr \le
2kE'_{s,x}\Big(\mathbf{1}_{\{\zeta>T-\tau(0)\}}\varphi(\bfX_{T-\tau(0)})
+\tdw(s,x)+ \int_{0}^{\zeta_{\tau}}\!\!g(\bfX_{r})\,dr\Big),
\]
which when combined with (\ref{eq3.07}) yields (\ref{eq2.3}).
\smallskip\\
Step 2. We are going to show that $v_{n}$ satisfies all the
assertions of the proposition. To shorten notation, we write
$\overline{v}$ (resp. $\overline{\psi}$) instead of $v_n$ (resp.
$\psi_{n}$). Since $\varphi_m\in L^2(D)$ and $\overline{\phi}_m\in
L^2(D_T)$, it follows from Step 1 that there exists
$\overline{v}_{m}\in\mathcal{FS}^{2}$ such that
$\overline{v}_{m}\in \mathcal{W}(D_{T})$ and for q.e. $(s,x)\in
D_{T}$,
\[
\overline{v}_{m}(\bfX_{t})=\overline{Y}_{t}^{m},\quad t\in
[0,\zeta_{\tau}], \quad P'_{s,x} \mbox{-a.s.},
\]
\[
\sigma \nabla \overline{v}_{m}(\bfX)=\overline{Z}^{m}, \quad
dt\otimes P'_{s,x} \mbox{-a.s. on } [0,\zeta_{\tau}]\times\Omega'.
\]
Put $\overline{v}'(s,x)=\limsup_{m \rightarrow\infty}
\overline{v}_{m}(s,x)$, $(s,x)\in D_{T}$. Then by (\ref{eq2.6}),
$\overline{v}'(\bfX_{t})=\overline{Y}_{t}$, $t\in
[0,\zeta_{\tau}]$, $P'_{s,x}$-a.s. for q.e. $(s,x)\in D_{T}$. This
implies that $\overline{v}'\in\FF\mathcal{D}^{q}$ for $q \in
(0,1)$, $\overline{v}'$ is of class (FD) and
$\overline{v}'=\overline{v}$ q.e. The last statement implies that
$\overline{v}$ belongs to the same spaces of functions as
$\overline{v}'$. We have proved that (\ref{eq2.3}) is true if
$\mu=g\cdot m_{1}$ for some $g \in L^{2}(D_{T})$, $g\ge0$, and
$\varphi\in L^{2}(D)$. Therefore for q.e. $(s,x) \in D_{T}$,
\begin{align}
\label{eq2.8} &E'_{s,x}\int_{0}^{\zeta_{\tau}}|\sigma
\nabla T_{k}(\ovV_{m})|(\bfX_{r})\,dr \nonumber\\
&\qquad\le 4k
E'_{s,x}\Big(\mathbf{1}_{\{\zeta>T-\tau(0)\}}\varphi(\bfX_{T-\tau(0)})
+\int_{0}^{\zeta_{\tau}}\bar{\phi}_{m}(\bfX_{r})\,dr \Big)
\nonumber \\
&\qquad\le 4k
E'_{s,x}\Big(\mathbf{1}_{\{\zeta>T-\tau(0)\}}\varphi(\bfX_{T-\tau(0)})
+\int_{0}^{\zeta_{\tau}}\bar{\phi}(\bfX_{r})\,dr \Big).
\end{align}
From this and Lemma \ref{lm2.3},
\begin{equation}
\label{eq2.9} \int_{D_{T}}|\sigma
\nabla(T_{k}(\ovV_{m}))|^{2}\,dm_{1}\le 4k(\|\varphi\|_{L^{1}}+
\|\overline{\phi}_{m}\|_{L^{1}})\le
4k(\|\varphi\|_{L^{1}}+\|\bar{\phi}\|_{L^{1}}).
\end{equation}
Due to \cite{Petitta}, from (\ref{eq2.9}) it follows that
$\ovV\in\mathcal{T}_{2}^{0,1}\cap L^{q}(0,T;W_{0}^{1, q}(D))$ and
for some subsequence (still denoted by $n$), $\sigma
\nabla\ovV_{m}\rightarrow\sigma \nabla\ovV$ weakly in
$L^{q}(D_{T})$ for $q\in[1,\frac{d+2}{d+1})$. On the other hand,
by (\ref{eq2.7}) and Lemma \ref{lm2.4},
$\sigma\nabla\ovV_{m}\rightarrow \overline{\psi}$, $m_{1}$-a.e. on
$D_{T}$. Hence $\overline{\psi}=\sigma \nabla\ovV$, $m_{1}$-a.e.
on $D_{T}$. Summarizing, $\ovV$ satisfies all the
assertions of the proposition.
\smallskip\\
Step 3. Using (\ref{eq3.15})--(\ref{eq3.18}) one can show in the
same manner as in  Step 2 that $v$ satisfies all the assertions of
the proposition. The only difference from Step 2 is that in
estimates of $\sigma\nabla T_k(v_n)$ of the form (\ref{eq2.8}),
(\ref{eq2.9}) the function $\bar{\phi}_{m}$ is replaced by
$\phi_{n}$, so to obtain the  estimates for $\sigma\nabla
T_k(v_n)$ which do not depend on $n$  we have to show that
\begin{equation}
\label{eq2.11}
E'_{s,x}\int_{0}^{\zeta_{\tau}}\phi_{n}(\mathbf{X}_{r})\,dr \le
E'_{s,x}\int_{0}^{\zeta_{\tau}}dA_{r}^{\mu}
\end{equation}
for q.e. $(s,x)\in D_{T}$ and that $\|\phi_{n}\|_{TV}\le
\|\mu\|_{TV}$. But (\ref{eq2.11}) follows from the fact that
$v_{n}(s,x)\le v(s,x)$ for q.e. $(s,x)\in D_T$ and the estimate
for $\|\phi_{n}\|_{TV}$ follows from (\ref{eq2.11}) and Lemma
\ref{lm2.3}.
\end{dow}

\nsubsection{Cauchy-Dirichlet problem}
\label{sec4}

It is convenient to begin the study of the obstacle problem
(\ref{eq1.1}) with the study of the Cauchy-Dirichlet problem
\begin{equation}
\label{eq2.C} \left\{
\begin{array}{l}\frac{\partial u}{\partial t}+A_{t}u
=-f_{u}-\mu, \smallskip \\
u(T,\cdot)=\varphi,\quad u(t,\cdot)_{|\partial D}=0,\quad t\in
(0,T),
\end{array}
\right.
\end{equation}
which can be regarded as problem (\ref{eq1.1}) with
$h_1\equiv-\infty$, $h_2\equiv+\infty$.

Let us recall that every functional $\Phi\in \mathcal{W}'(D_{T})$
admits decomposition of the form
\begin{equation}
\label{eq4.01} \Phi=(g)_{t}+\mbox{div}(G)+f,
\end{equation}
where $g \in L^{2}(0,T; H^1_0(D))$, $G=(G^1,\dots,G^d)$, $f \in
L^{2}(D_{T})$, i.e. for every $\eta\in\mathcal{W}(D_{T})$,
\[
\Phi(\eta)=-\langle g, \frac{\partial\eta}{\partial t}
\rangle-\langle G,\nabla\eta\rangle_{L^{2}} + \langle f, \eta
\rangle_{L^{2}}\,,
\]
where $\langle\cdot,\cdot\rangle$ denotes the duality between
$L^{2}(0,T;H_{0}^{1}(D))$ and $L^{2}(0,T;H^{-1}(D))$. It is also
known that every measure $\mu\in\MM_{0,b}(D_{T})$ admits
decomposition of the form
\begin{equation}
\label{eq4.1} \mu=\Phi+f,
\end{equation}
where $\Phi \in \mathcal{W}'(D_{T})$, $f \in L^{1}(D_{T})$, i.e.
for every $\eta \in C_{0}^{\infty}(D_{T})$,
\[
\int_{D_{T}}\eta \,d\mu= \Phi(\eta)+\int_{D_{T}}f\eta\,dm_1.
\]
Accordingly, $\mu\in\MM_{0,b}(D_{T})$ can be written in the form
\begin{equation}
\label{eq4.2} \mu=-(g)_{t} + \mbox{div}(G)+f
\end{equation}
for some $g\in L^{2}(0,T;H^1_0(D))$, $G\in L^{2}(D_{T})$ and $f\in
L^{1}(D_{T})$. Let us stress that in general, $\Phi,f$ of the
decomposition (\ref{eq4.1}) cannot be taken nonnegative even if
$\mu$ is nonnegative.

We say that a triple $(g,G,f)$ is the decomposition of
$\mu\in\MM_{0,b}(D_{T})$ if (\ref{eq4.2}) is satisfied.

Let $\mu \in \MM_{0,b}(D_{T})$, $\varphi\in L^{1}(D)$ and let
$f:D_T\times\mathbb{R}\rightarrow\mathbb{R} $ be a Carath\`eodory
function.

\begin{df}
We say that a measurable function $u: D_{T} \rightarrow
\mathbb{R}$ is a renormalized solution of the problem
(\ref{eq2.C}) if
\begin{enumerate}
\item[(a)]$f_u\in L^{1}(D_{T})$,
\item[(b)]For some decomposition $(g,G,f)$ of $\mu$,
$u-g \in L^{\infty}(0,T; L^{2}(D))$, $T_{k}(u-g)\in
L^{2}(0,T;H^{1}_0(D))$ for $k\ge0$ and
\[
\lim_{k\rightarrow+\infty}\int_{\{k\le |u-g|\le k+1\}}|\nabla
u|\,dm_{1}=0,
\]
\item[(c)]For any $S \in W^{2, \infty}(\BR)$ with compact support,
\begin{align*}
&\frac{\partial}{\partial t}(S(u-g))+ \mbox{div}(a\nabla u
S'(u-g)) -S''(u-g)a\nabla u \cdot \nabla(u-g) \\
&\qquad=-S'(u-g)f-\mbox{div}(GS'(u-g))+GS''(u-g)\cdot\nabla(u-g)
\end{align*}
in the sense of distributions,
\item[(d)] $T_{k}(u-g)(T)=T_{k}(\varphi)$ in $L^{2}(D)$ for all $k \ge0$.
\end{enumerate}
\end{df}

Note that a different but equivalent definition of renormalized
solution of (\ref{eq2.C}) is given in \cite[Definition 4.1]{PPP}.

Set
\[
\Theta_{k}(s)=\int_{0}^{s}T_{k}(y)\,dy,\quad s\in\BR
\]
and
\[
E=\{\eta \in L^{2}(0, T; H^{1}_{0}(D)) \cap L^{\infty}(D_{T}):
\frac{\partial \eta}{\partial t}\in L^{2}(0, T;
H^{-1}(D))+L^{1}(D_{T})\}.
\]

\begin{df}
We say that a measurable function $u: D_{T} \rightarrow
\mathbb{R}$ is an entropy solution of (\ref{eq2.C}) if $f_{u}\in
L^{1}(D_{T})$, for some decomposition $(g, G, f)$ of $\mu$,
$T_{k}(u-g)\in L^{2}(0,T;H^{1}_{0}(D))$ for any $k\ge 0$,
$[0,T]\ni t\mapsto\int_{D}\Theta_{k}(u-g-\eta)(t,\cdot)\,dm$ is
continuous for any $\eta\in E$, $k\ge 0$, and
\begin{align}
\label{eq2.E} \nonumber & \int_{D}\Theta_{k}(u-g-\eta)(0,
\cdot)\,dm - \int_{D}\Theta_{k}(\varphi - \eta(T, \cdot))\,dm - \langle
\eta_{t}, T_{k}(u-g-\eta) \rangle \\& \qquad + \int_{D_{T}}a\nabla
u \cdot \nabla T_{k}(u-g-\eta)\,dm_{1} \nonumber \le
\int_{D_{T}}fT_{k}(u-g-\eta)\,dm_{1} \\& \qquad -
\int_{D_{T}}G\cdot \nabla T_{k}(u-g-\eta)\,dm_{1}
+\int_{D_{T}}f_{u}T_{k}(u-g-\eta)\,dm_{1}.
\end{align}
\end{df}

\begin{uw}
\label{uw4.1} (i) From  \cite[Theorem 3.1]{DP} it follows that $u$
is a renormalized solution of (\ref{eq2.C}) iff it is an entropy
solution of (\ref{eq2.C}).
\smallskip\\
(ii) If $u$ is a renormalized solution of (\ref{eq2.C}) then it is
a distributional solution of (\ref{eq2.C}) in the sense that
$u,\nabla u\in L^1(D_T)$ and for any $\eta\in C^{\infty}_0(D_T)$,
\begin{equation}
\label{eq4.6} \int_{D_T}u\frac{\partial\eta}{\partial t}\,dm_1
+\int_{D_T}a\nabla u\cdot\nabla\eta\,dm_1
=\int_D\varphi\eta(T,\cdot)\,dm
+\int_{D_T}f_u\eta\,dm_1+\int_{D_T}\eta\,d\mu
\end{equation}
(see Proposition 4.5 and Theorem 4.11 in \cite{PPP}).
\end{uw}

\begin{lm}
\label{lm2.5} Assume that $\mu_{n},\mu\in\mathcal{M}_{0,b}(D_{T})$
and $\|\mu_{n}-\mu\|_{TV} \rightarrow 0$. Then there exist $g_{n},
g\in L^{2}(0, T; H_{0}^{1}(D))$, $G_{n}, G \in L^{2}(D_{T})$,
$f_{n},f\in L^{1}(D_{T})$ such that
\begin{equation}
\label{eq4.7} \mu_{n}=(g_{n})_{t}+\mbox{\rm div}(G_{n}) + f_{n},
\quad \mu=(g)_{t}+\mbox{\rm div}(G)+f
\end{equation}
and
\begin{equation}
\label{eq4.08} G_{n} \rightarrow G\mbox{ in }L^{2}(D_{T}), \quad
f_{n} \rightarrow f\mbox{ in }L^{1}(D_{T}), \quad g_{n}
\rightarrow g\mbox{ in }L^{2}(0, T; H^{1}_{0}(D)).
\end{equation}
\end{lm}
\begin{dow}
From the proof of \cite[Theorem 2.7]{DPP} it follows that each
$\mu\in \mathcal{M}_{0,b}(D_{T})$ admits a decomposition of the
form (\ref{eq4.1}) with $\Phi,f$ such that
$\|\Phi\|_{\mathcal{W'}(D_{T})}\le 1$, $\|f\|_{L^{1}}\le
\|\mu\|_{TV}$.  Moreover, by \cite[Lemma 2.24]{DPP}, $\Phi$ admits
decomposition (\ref{eq4.01}) with $g,G,h$ such that
\[
\|g\|_{L^{2}(0,T;H^{1}_{0}(D))} + \|G\|_{L^{2}} + \|h\|_{L^{2}}\le
\|\Phi\|_{\mathcal{W}'(D_{T})}.
\]
Therefore repeating arguments from the proof of \cite[Corollary
3.2]{Leone} we get the desired result.
\end{dow}

\begin{lm}
\label{lm2.6} Let $\{\mu_{n}\} \subset \mathcal{M}_{0,b}(D_{T})$,
$\mu \in\mathcal{M}_{0,b}(D_{T})$, $\{\varphi_{n}\} \subset
L^{1}(D)$, $\varphi \in L^{1}(D)$ and let $u_{n}$ (resp. u) be a
renormalized solution of \mbox{\rm(\ref{eq2.C})} with terminal
condition $\varphi_{n}$ (resp. $\varphi$), $f\equiv 0$ and with
$-\mu_{n}$ (resp. $-\mu$) on the right-hand side. If
$\|\mu_{n}-\mu\|_{TV} \rightarrow 0$ and
$\|\varphi_{n}-\varphi\|_{L^{1}} \rightarrow 0$ then $u_{n}
\rightarrow u$, $m_{1}$-a.e.
\end{lm}
\begin{dow}
By Lemma  \ref{lm2.5} we may assume that $\mu_n,\mu$ are given by
(\ref{eq4.7}) and (\ref{eq4.08}) is satisfied. But then the lemma
follows from  \cite[Proposition 4]{Petitta}.
\end{dow}

\begin{stw}
\label{stw2.3} Let $\varphi \in L^{1}(D)$, $\mu \in
\MM_{0,b}(D_{T})$ and let $v$ be defined by
\mbox{\rm(\ref{eq2.2})}. Then $v\in\mathcal{T}_{2}^{0,1}$, $v\in
L^{q}(0,T;W_{0}^{1,q}(D))$ for $q\in[1,\frac{d+2}{d+1})$ and $v$
is a renormalized solution of the problem
\begin{equation}
\label{eq2.12}
\left\{
\begin{array}{l}\frac{\partial v}{\partial t}+A_{t}v=-\mu,\smallskip\\
v(T,\cdot)=\varphi, \quad v(t,\cdot)_{|\partial D}=0,\quad t\in
(0,T).
\end{array}
\right.
\end{equation}
\end{stw}
\begin{dow}
Without loss of generality we may assume that $\varphi \ge 0$ and
$\mu \ge 0$. Assume for a moment that $\varphi \in L^{2}(D_{T})$
and $\mu \in \MM_{0,b}^{+} \cap \mathcal{W}'(D_{T})$. From the
proof of Proposition \ref{stw3.7} it follows that $v$ is the q.e.
limit of $v_{n}$, where $v_{n}$ is a weak solution of
\[
\left\{
\begin{array}{l}\frac{\partial v_{n}}{\partial t}
+A_{t}v_{n}=-n(v_{n}-v)^{-}, \smallskip\\
v_{n}(T,\cdot)=\varphi, \quad v_{n}(t, \cdot)_{|\partial D}=0,
\quad t \in (0,T).
\end{array}
\right.
\]
But it is known (see, e.g., \cite[Theorem 1.1]{MP}) that
$\{v_{n}\}$ converges in $L^{2}(D_{T})$ to a unique weak solution
of (\ref{eq2.12}). Therefore $v$ is a weak solution of
(\ref{eq2.12}). Since a weak solution of (\ref{eq2.12}) is a
renormalized solution, this proves the proposition under the
additional assumptions on $\varphi,\mu$. Assume now that
$\varphi\in L^{1}(D)$ is nonnegative and
$\mu\in\MM^+_{0,b}(D_{T})$. By \cite[Theorem 5.6]{Oshima} there
exists a generalized nest, i.e. an ascending sequence $\{F_{n}\}$
of compact subsets of $D_{T}$ such that cap$(K\setminus F_{n})
\rightarrow 0$ for every compact $K \subset D_{T}$, with the
property that $\mathbf{1}_{F_{n}}\cdot\mu\in\MM^+_{0,b}(D_{T})\cap
\mathcal{W}'(D_{T})$ and $\mu(D_{T}\setminus\bigcup_{n}F_{n})=0$.
Let $\varphi_{n}=\varphi \wedge n$. By what has already been
proved, $v_n$ defined as
\begin{equation}
\label{eq.KF}
v_{n}(s,x)= E'_{s,x}\Big(\mathbf{1}_{\{\zeta>T-\tau(0)\}}
\varphi_{n}(\bfX_{T-\tau(0)}) +
\int_{0}^{\zeta_{\tau}}\mathbf{1}_{F_{n}}(\bfX_{r})dA_{r}^{\mu}\Big),
\quad(s,x)\in D_T
\end{equation}
is a renormalized solution of (\ref{eq2.12}) with $\varphi, \mu$
replaced by $\varphi_{n}$ and $\mu_{n}$, respectively. Since
$\|\varphi_{n}-\varphi\|_{L^{1}} \rightarrow 0$ and $\{F_{n}\}$ is
a generalized nest, we conclude from (\ref{eq.KF}) that
$v_{n}(s,x)\rightarrow v(s,x)$ for q.e. $(s,x)\in D_{T}$. This
completes the proof because by Lemma \ref{lm2.6}, $\{v_{n}\}$
converges to the renormalized solution of (\ref{eq2.12}).
\end{dow}

\begin{tw}
\label{tw3.1} Assume \mbox{\rm(H1)--(H5)}. Then there exists a
unique renormalized solution $u$ of \mbox{\rm(\ref{eq2.C})}.
Moreover, $u\in\FF\mathcal{D}$, $u\in L^{q}(0,T;W_{0}^{1,q}(D))$
for $q\in[1,\frac{d+2}{d+1})$ and
\begin{equation}
\label{eq4.8} u(s,x)=E'_{s,x}\Big(\mathbf{1}_{\{\zeta>T-\tau(0)\}}
\varphi(\bfX_{T-\tau(0)}) +
\int_{0}^{\zeta_{\tau}}f_{u}(\bfX_{r})\,dr
+\int_{0}^{\zeta_{\tau}}dA_{r}^{\mu}\Big)
\end{equation}
for q.e. $(s,x) \in D_{T}$. Finally, there exists $C>0$ depending
only on $\kappa,T$ such that
\begin{equation}
\label{eq4.9} \|f_u\|_{L^1(D_T)}\le C(\|\varphi\|_{L^1(D)}
+\|f(\cdot,\cdot,0)\|_{L^1(D_T)} +\|\mu\|_{TV}).
\end{equation}
\end{tw}
\begin{dow}
By Lemma \ref{lm3.1} and \cite[Proposition 3.10]{Kl:BSM}, for q.e.
$(s,x)\in D_{T}$ there is a unique solution $(Y^{s,x}, Z^{s,x})$
of BSDE$_{s,x}(\varphi,D,f+d\mu)$. By Proposition \ref{stw3.5},
for q.e. $(s,x)\in D_T$ we have $Y^{s,x}_t=u(\bfX_t)$,
$P'_{s,x}$-a.s. for every $t\in[0,T-\tau(0)]$, where
$u:D_T\rightarrow\BR$ satisfies (\ref{eq4.8}). Furthermore, by
Lemma \ref{lm2.3} and (\ref{eq3.03}), $f_u$ satisfies
(\ref{eq4.9}). Hence $f_u\cdot m_1+\mu\in\MM_{0,b}(D_T)$ and from
Proposition \ref{stw2.3} it follows that $u$ is a renormalized
solution of (\ref{eq2.C}) and has the desired regularity
properties. The uniqueness part of the theorem follows from the
uniqueness of solutions of BSDE$_{s,x}(\varphi,D,f+d\mu)$.
\end{dow}

\begin{uw}
\label{uw4.6} Under the assumptions of Theorem \ref{tw3.1}, for
every $q\in[1,\frac{d+2}{d+1})$  there exist $C,\gamma>0$ such
that
\[
\|u\|_{L^q(0,T;W^{1,q}_0(D))}\le
C(\|\varphi\|_{L^1(D)}+\|f(\cdot,\cdot,0)\|_{L^1(D_T)}
+\|\mu\|_{TV})^{\gamma},
\]
because from Remark \ref{uw4.1}(ii) and results proved in
\cite[Section 3]{BDGO} (see also \cite{BG}) it follows that if $u$
is a solution of (\ref{eq2.C}) then
$\|u\|_{L^q(0,T;W^{1,q}_0(D))}\le c\|f_u\cdot m_1+
\mu\|_{TV}^{\gamma}$ for some $c,\gamma>0$.
\end{uw}

\nsubsection{Obstacle problem}
\label{sec5}

We begin with a probabilistic definition of a solution of the
obstacle problem.

\begin{df}
Assume \mbox{\rm(H1), (H4)} and let $h_{1}, h_{2}$ be measurable
functions on $D_{T}$ such that $h_{1}\le h_{2}$, $m_{1}$-a.e. We
say that a pair $(u, \nu)$ consisting of a measurable function $u:
D_{T} \rightarrow \BR$ and a measure $\nu$ on $D_{T}$ is a
solution of the obstacle problem with terminal condition
$\varphi$, right-hand side $f+d\mu$ and  obstacles $h_{1}, h_{2}$
(OP$(\varphi, f+d\mu,h_{1},h_{2})$ for short) if
\begin{enumerate}
\item[(a)] $f_{u}\in L^{1}(D_{T})$, $\nu \in \MM_{0,b}(D_{T})$,
$h_{1}\le u \le h_{2}$, $m_{1}$-a.e.,
\item[(b)]For q.e. $(s,x)\in D_{T}$,
\[
u(s,x)= E'_{s,x}\Big(\mathbf{1}_{\{\zeta>T-\tau(0)\}}
\varphi(\bfX_{T-\tau(0)})
+\int_{0}^{\zeta_{\tau}}f_{u}(\bfX_{r})\,dr
+\int_{0}^{\zeta_{\tau}}d(A_{r}^{\mu}+A_{r}^{\nu})\Big),
\]
\item[(c)]For every $h_{1}^{*}, h_{2}^{*} \in \FF\mathcal{D}$ such that
$h_{1} \le h_{1}^{*} \le u \le h_{2}^{*} \le h_{2}$, $m_{1}$-a.e.
we have
\[
\int_{0}^{\zeta_{\tau}}
(u_{-}(\bfX_{r})-h_{1-}^{*}(\bfX_{r}))\,dA_{r}^{\nu^{+}}
=\int_{0}^{\zeta_{\tau}}
(h_{2-}^{*}(\bfX_{r})-u_{-}(\bfX_{r}))\,dA_{r}^{\nu^{-}}=0, \quad
P'_{s,x}\mbox{\rm-a.s.}
\]
for q.e. $(s,x)\in D_{T}$, where
$g_{-}(\bfX_r)=\lim_{t<r,t\rightarrow r}g(\bfX_t)$ for $g:=
u,h^{*}_1,h^{*}_2$.
\end{enumerate}
\end{df}

We say that $(u,\nu)$ is a solution of the obstacle problem with
one lower (resp. upper) barrier $h$
($\underline{\mbox{O}}$P$(\varphi, f+d\mu, h)$ (resp.
$\overline{\mbox{O}}$P$(\varphi, f+d\mu, h)$) for short) if $(u,
\nu)$ satisfies the conditions of the above definition with
$h_{1}=h$, $h_{2}=+\infty$ (resp. $h_{1}=-\infty, h_{2}=h$) and
$\nu\in\MM_{0,b}^{+}(D_{T})$ (resp. $-\nu \in
\MM_{0,b}^{+}(D_{T})$).

\begin{uw}
Let us note that in view of Proposition \ref{stw2.3}, condition
(b) in the above definition says that $u$ is a renormalized
solution of (\ref{eq2.C}) with $\mu$ replaced by $\mu+\nu$.
Condition (c) provides a probabilistic formulation of minimality
of $\nu$. In \cite{Pierre1}  minimality of $\nu$ was expressed by
using the notion of precise versions of functions introduced in
\cite{Pierre}.  In fact, in the linear case with $L^{2}$ data,
condition (c) coincides with that introduced in \cite{Pierre1},
because for every parabolic potential $h$,
$\hat{h}(\mathbf{X}_{t})=h_{-}(\mathbf{X}_{t}),\, t\in
[0,\zeta_{\tau}]$, where $\hat{h}$ is the precise version of $h$
(see \cite{Kl:SPA} for details).
\end{uw}

As in the case of the Cauchy-Dirichlet problem considered in the
previous section, the proof of the existence of a solution of the
obstacle problem relies heavily on the results on BSDEs proved in
\cite{Kl:BSM}.

Suppose we are given a filtered probability space and $A,B,
\sigma,\xi$ as in Subsection \ref{sec3.1}. Moreover, suppose that
we are given two progressively measurable processes $U,L$ such
that $U\le L$ and $f:\Omega\times\BR_+\times\BR\rightarrow\BR$
such that $f(\cdot,y)$ is progressively measurable ($f$ does not
depend on $z$).

\begin{df}
A triple $(Y,Z,R)$ of progressively measurable processes is a
solution of the reflected backward stochastic differential
equation with terminal condition $\xi$, right-hand side $f+dA$ and
two reflecting barriers $L,U$ (RBSDE$(\xi,\sigma,f+dA,L,U)$ for
short) if
\begin{enumerate}
\item[(a)] $Z \in M$, $t\mapsto f(t,Y_{t},Z_{t})\in
L^{1}(0,\sigma)$, $P$-a.s.,
\item[(b)]$L_{t}\le Y_{t}\le U_{t}$ for a.e. $t\in [0,\sigma]$,
$P$-a.s.,
\item[(c)]$Y_{t}= \xi+\int_{t}^{\sigma}f(s,Y_{s},Z_{s})\,ds$
$+\int_{t}^{\sigma}dA_{s}-\int_{t}^{\sigma}dR_{s}
-\int_{t}^{\sigma}Z_{s}\,dB_{s}$, $0\le t\le\sigma$, $P$-a.s.,
\item[(d)]$R\in\mathcal{V}$ and for any $\hat{L},\check{U}\in\mathcal{D}$
such that $L_{t}\le\hat{L}_{t}\le Y_{t}\le\check{U}_{t}\le U_{t}$
for a.e. $t\in[0,\sigma]$,
$\int_{0}^{\sigma}(Y_{t-}-\hat{L}_{t-})\,dR_{t}^{+}
=\int_{0}^{\sigma}(\check{U}_{t-}-Y_{t})\,dR^{-}_{t}=0$, $P$-a.s.
\end{enumerate}
\end{df}

We say that $(Y,Z,R)$ is a solution of RBSDE with one lower (resp.
upper) barrier $L$ (resp. $U$), terminal condition $\xi$, the
right-hand side $f+dA$
($\underline{\mbox{R}}$BSDE$(\xi,\sigma,f+dA,L)$ (resp.
$\overline{\mbox{R}}$BSDE$(\xi,\sigma,f+dA,U)$) for short) if
$(Y,Z,R)$ satisfies the conditions of the above definition with
$U\equiv+\infty$ (resp. $L \equiv-\infty$) and
$R\in\mathcal{V}^{+}$ (resp. $-R\in\mathcal{V}^{+}$).

To shorten notation, in what follows for given $\mu,\varphi, f$
and $h_{1}, h_{2}$ we denote by RBSDE$_{s,x}(\varphi, D,f+d\mu,
h_{1},h_{2})$ the reflected BSDE with data
$\mathbf{1}_{\{\zeta>T-\tau(0)\}}\varphi(\bfX_{T-\tau(0)})$,
$\zeta_{\tau}$, $f(\bfX, \cdot, \cdot)+dA^{\mu}$,
$h_{1}(\bfX),h_{2}(\bfX)$ considered on the space
$(\Omega',\GG'_{\infty},\{\GG'_t\},P'_{s,x})$.

Let $\MM^q$, $q>0$, denote the space of continuous martingales
such that $E([M]_T)^{q/2}<+\infty$ for every $T>0$.

We will need the following general growth conditions for $f$.
\begin{enumerate}
\item[(A6)]There exists a semimartingale $\Gamma$ such that
$\Gamma$ is of class (D), $\Gamma\in\MM^{q}\oplus\mathcal{V}^{1}$
for $q \in (0,1)$, $L_{t}\le\Gamma_{t}$ for a.e. $t\in [0,\sigma]$
and $E\int_{0}^{\sigma}f^{-}(t,\Gamma_{t})\,dt <+\infty$,
\item[(A6$'$)]There exists a semimartingale $\Gamma$ such that
$\Gamma$ is of class (D), $\Gamma\in\MM^{q}\oplus\mathcal{V}^{1}$
for $q \in (0,1)$, $L_{t}\le\Gamma_{t}\le U_{t}$ for a.e. $t\in
[0,\sigma]$ and $E\int_{0}^{\sigma}|f(t,
\Gamma_{t})|\,dt<+\infty$.
\end{enumerate}

It is known  that under (A1)--(A6) (resp. (A1)--(A5), (A6$'$))
there exists a unique solution of BSDE$(\xi,\sigma,f+dV,L)$ (resp.
RBSDE $(\xi,\sigma,f+dA,L,U)$) (see  Theorems 5.7, 6.6 in
\cite{Kl:BSM}). By Lemma \ref{lm3.1}, assumptions (H1)--(H5) imply
that (A1)--(A5) hold under the measure $P'_{s,x}$ for q.e.
$(s,x)\in D_{T}$.

Analytic analogue of conditions (A6), (A6$'$) is as follows.

\begin{enumerate}
\item[(H6)]There exists a measurable function $v:D_T\rightarrow\BR$, a
measure $\lambda\in\MM_{0,b}(D_{T})$ and $\phi\in L^{1}(D)$, $\phi
\ge\varphi$, such that $v$ is a renormalized solution of the
problem
\begin{equation}
\label{eq3.1}
\left\{
\begin{array}{l}\frac{\partial v}{\partial t}+A_{t}v=-\lambda, \\
 v(T,\cdot)=\phi, \quad v(t, \cdot)_{|\partial D}=0, \quad t\in(0,T)
\end{array}
\right.
\end{equation}
and $f^{-}_{v}\in L^{1}(D_{T})$, $v\ge h_{1}$, $m_{1}$-a.e. on
$D_{T}$,
\item[(H6$'$)] There exists a measurable function $v:D_T\rightarrow\BR$,
a measure $\lambda\in\MM_{0,b}(D_{T})$ and $\phi\in L^{1}(D)$,
$\phi\ge\varphi$, such that $v$ is a renormalized solution of
(\ref{eq3.1}) and $f_{v}\in L^{1}(D_{T})$, $h_{1}\le v\le h_{2}$,
$m_{1}$-a.e. on $D_{T}$.
\end{enumerate}

\begin{lm}
\label{lm3.2} Let $L=h_1(\bfX), U=h_2(\bfX)$. If $v$ satisfies
\mbox{\rm(H6)} (resp. \mbox{\rm(H6$'$)}) then $\Gamma=v(\bfX)$
satisfies \mbox{\rm(A6)} (resp. \mbox{\rm(A6$'$)}) under the
measure $P'_{s,x}$ for q.e. $(s,x)\in D_{T}$.
\end{lm}
\begin{dow}
Follows immediately from Propositions \ref{stw3.7} and
\ref{stw2.3}.
\end{dow}
\medskip

We first prove the comparison and uniqueness results for solutions
of the obstacle problem.

\begin{stw}
\label{stw3.1} Let  $(u_{i},\nu_{i})$ be a solution of \mbox{\rm
OP}$(\varphi_{i},f^{i}+d\mu_{i},h_{1}^{i}, h_{2}^{i})$, $i=1,2$.
If $\varphi_{1}\le\varphi_{2}$, $m$-a.e., $d\mu_{1}\le d\mu_{2}$,
$h_{1}^{1}\le h_{1}^{2}$, $h_{2}^{1}\le h_{2}^{2}$, $m_{1}$-a.e.
and either
\[
f^{2} \mbox{ satisfies \mbox{\rm(H3)} and }
\mathbf{1}_{\{u_{1}>u_{2}\}}(f_{u_{1}}^{1}-f_{u_{2}}^{2})\le0,
\quad m_{1}\mbox{-a.e.}
\]
or
\[
f^{1} \mbox{ satisfies \mbox{\rm(H3)} and }
\mathbf{1}_{\{u_{1}>u_{2}\}}(f_{u_{2}}^{1}-f_{u_{2}}^{2})\le0,
\quad m_{1}\mbox{-a.e.}
\]
then $u_{1} \le u_{2}$ q.e.
\end{stw}
\begin{dow}
The desired result follows from \cite[Corollary 6.2]{Kl:BSM},
because by the definition of a solution of the obstacle problem
and Proposition \ref{stw3.7},  for q.e. $(s,x)\in D_{T}$ the
triple $(u_{i}(\bfX),\sigma\nabla u_{i}(\bfX), A^{\nu_{i}})$ is a
solution of RBSDE$_{s,x}(\varphi_{i}, D, f^{i}+d\mu_{i},
h_{1}^{i}, h_{2}^{i})$.
\end{dow}

\begin{wn}
Let assumption  \mbox{\rm(H3)} hold. Then there exists at most one
solution of \mbox{\rm OP}$(\varphi,f+d\mu,h_{1},h_{2})$.
\end{wn}
\begin{dow}
Follows immediately from Proposition \ref{stw3.1}.
\end{dow}
\medskip

Before proving our main result on existence and approximation by
the penalization method of solutions of the obstacle problem with
one barrier let us recall  that a function $v$ on $D_{T}$ is a
supersolution of PDE$(\varphi, f+d\mu)$ if there exists a measure
$\lambda\in\MM_{0,b}^{+}(D_{T})$ such that $v$ is a renormalized
solution of the problem
\[
\left\{
\begin{array}{l}\frac{\partial v}{\partial t}+A_{t}v
=-f_{v}-\mu-\lambda, \smallskip\\
v(T,\cdot)=\varphi, \quad v(t,\cdot)_{|\partial D}=0, \quad t\in
(0,T).
\end{array}
\right.
\]

\begin{tw}
\label{tw5.5} Assume \mbox{\rm(H1)--(H6)}.
\begin{enumerate}
\item[\rm(i)]There exists a solution $(u,\nu)$ of
\mbox{\rm OP}$(\varphi,f+d\mu,h_{1})$ such that $u$ is of class
\mbox{\rm(FD)}, $u\in \FF\mathcal{D}^{q}$ for $q\in(0,1)$, $\nabla
u\in FM^{q}$ for $q\in(0,1)$, $u\in\mathcal{T}_{2}^{0,1}$, $u\in
L^{q}(0,T;W_{0}^{1,q}(D))$ for $q\in[1,\frac{d+2}{d+1})$.
\item[\rm(ii)]For $n\in\BN$ let $u_{n}$ be
a renormalized solution of the problem
\begin{equation}
\label{eq3.3}
\left\{
\begin{array}{l}\frac{\partial u_{n}}{\partial t}+A_{t}u_{n}
=-f_{u_{n}}-n(u_{n}-h_{1})^{-}-\mu, \smallskip\\
u_{n}(T, \cdot)=\varphi, \quad u_{n}(t,\cdot)_{|\partial D}=0,
\quad t\in(0, T).
\end{array}
\right.
\end{equation}
Then $u_{n}\nearrow u$ q.e. on $D_{T}$, $u_{n}\rightarrow u$ in
$L^q(0,T;W^{1,q}_0(D))$ for any $q\in[1,\frac{d+2}{d+1})$ and
$\nu_{n}\rightarrow\nu$ weakly, where
$d\nu_{n}=n(u_{n}-h_{1})^{-}dm_{1}$.
\end{enumerate}
\end{tw}
\begin{dow}
(i) By Lemmas \ref{lm3.1}, \ref{lm3.2} and \cite[Theorem
5.7]{Kl:BSM}, for q.e. $(s,x)\in D_{T}$ there exists a unique
solution $(Y^{s,x},Z^{s,x},R^{s,x})$ of
$\underline{\mbox{R}}$BSDE$_{s,x}(\varphi,D,f+d\mu, h_{1})$ such
that $Y^{s,x}$ is of class (D),
$(Y^{s,x},Z^{s,x})\in\mathcal{D}^{q}\otimes M^{q}$ for $q\in(0,1)$
and $R^{s,x} \in \mathcal{V}^{+,1}$. By \cite[Theorem
5.7]{Kl:BSM}, for q.e. $(s,x)\in D_{T}$,
\begin{equation}
\label{eq3.4} Y_{t}^{n, s, x} \nearrow Y_{t}^{s,x}, \quad
t\in[0,\zeta_{\tau}],\,P'_{s,x}\mbox{-a.s.},
\end{equation}
\begin{equation}
\label{eq3.5} Z^{n,s,x} \rightarrow Z^{s,x}, \quad dt\otimes
P'_{s,x}\mbox{-a.e. on }[0,\zeta_{\tau}]\times\Omega',
\end{equation}
where $(Y^{n,s,x},Z^{n,s,x})$ is a solution of the equation
\begin{align*}
Y_{t}^{n,s,x}=&\mathbf{1}_{\{\zeta>T-\tau(0)\}}\varphi(\bfX_{T-\tau(0)})
+\int_{t}^{\zeta_{\tau}}(f(r,\bfX_{r},Y_{r}^{n,s,x})
+n(Y_{r}^{n,s,x}-h_{1}(\bfX_{r}))^{-})\,dr\\
& +\int_{t}^{\zeta_{\tau}}dA_{r}^{\mu}
-\int_{t}^{\zeta_{\tau}}Z_{r}^{n,s,x}\,dB_{r},\quad t \in
[0,\zeta_{\tau}],\quad P'_{s,x}\mbox{-a.s.}
\end{align*}
By Proposition \ref{stw3.7}, $u_{n}(\bfX_{t})=Y_{t}^{n,s,x}$,
$t\in[0,\zeta_{\tau}]$, $Z^{n,s,x}=\sigma\nabla u_{n}(\bfX)$,
$dt\otimes P'_{s,x}$-a.e. on $[0,\zeta_{\tau}]\times\Omega'$ for
q.e. $(s,x)\in D_{T}$. For $(s,x)\in D_{T}$ set
$u(s,x)=\limsup_{n\rightarrow +\infty}u_{n}(s,x)$. Then by
(\ref{eq3.4}), for q.e. $(s,x)\in D_{T}$,
\begin{equation}
\label{eq5.5} u(\bfX_{t})=Y_{t}^{s,x}, \quad t\in[0,\zeta_{\tau}],
\,P'_{s,x}\mbox{-a.s.}
\end{equation}
Moreover, by (\ref{eq3.5}) and Lemma \ref{lm2.4} there exists a
Borel measurable function $\psi$ on $D_{T}$ such that for q.e.
$(s,x)\in D_{T}$,
\begin{equation}
\label{eq5.6} \psi(\bfX)=Z^{s,x}, \quad dt\otimes
P'_{s,x}\mbox{-a.e. on } [0,\zeta_{\tau}]\times\Omega'.
\end{equation}
Set
\[
A_{t}=u(\bfX_{0})-u(\bfX_{t})-\int_{0}^{t}f_{u}(\bfX_{r})\,dr
-\int^t_0dA^{\mu}_r +\int_{0}^{t}\psi(\bfX_{r})\,dB_{r}, \quad
t\in[0, \zeta_{\tau}].
\]
By (\ref{eq5.5}) and (\ref{eq5.6}),
\begin{equation}
\label{eq5.07} A_{t}=R_{t}^{s,x}, \quad t\in[0,\zeta_{\tau}],\,
P'_{s,x}\mbox{-a.s.}
\end{equation}
for q.e. $(s,x)\in D_{T}$. It is an elementary check that $A$ is
an AF of $\BX'$. In fact, by (\ref{eq5.07}), it is a positive
functional. Therefore by Proposition in Section II.1 of
\cite{Revuz} and \cite[Theorem 5.6]{Oshima} there exists a
positive smooth measure $\nu$ on $D_{T}$ such that $A=A^{\nu}$.
From this, the definition of $A$ and (\ref{eq5.5}) it follows that
\begin{equation}
\label{eq5.7} u(s,x)=E'_{s,x}\Big(\mathbf{1}_{\{\zeta>T-\tau(0)\}}
\varphi(\bX_{T-\tau(0)})
+\int_{0}^{\zeta_{\tau}}f_{u}(\bfX_{r})\,dr
+\int_{0}^{\zeta_{\tau}}d(A_{r}^{\mu}+A^{\nu}_r)\Big)
\end{equation}
for q.e. $(s,x)\in D_T$. Let $v$ denote the function from
condition (H6). By Lemma \ref{lm3.2}, $v(\bfX)$ satisfies (A6)
(with $L=h_1(\bfX)$) under $P'_{s,x}$ for q.e. $(s,x)\in D_T$.
Therefore arguing as  at the beginning of the proof of
\cite[Theorem 5.7]{Kl:BSM} and using Theorem \ref{tw3.1} one can
show that there exists a supersolution $\bar v$ of
PDE$(\varphi\vee\phi,f+d\mu)$ such that $u_n\le\bar v$ q.e. on
$D_T$ for $n\in\BN$. From this, the fact that $u_{1}\le u$ q.e. on
$D_{T}$ and  (H3) it follows that $f_{\bar v}-\kappa \bar v\le
f_{u}-\kappa u\le f_{u_{1}}-u_1$, $m_{1}$-a.e. Hence $f_{u}\in
L^{1}(D_{T})$, because $u_1,\bar v,f_{u_1},f_{\bar v}\in
L^{1}(D_{T})$. Since $u \le\bar v$ q.e. on $D_{T}$, it follows
from (\ref{eq5.7}) that
\begin{align*}
E'_{s,x}\int_{0}^{\zeta_{\tau}}dA_{r}^{\nu} &\le
2E'_{s,x}|\varphi(\bfX_{T-\tau(0)})| +
E'_{s,x}|\phi(\bfX_{T-\tau(0)})|\\
&\quad+E'_{s,x}\int_{0}^{\zeta_{\tau}}(|f_{u}(\bfX_{r})| +|f_{\bar
v}(\bfX_{r})|)\,dr
+2E'_{s,x}\int_{0}^{\zeta_{\tau}}d(|A^{\mu}|_r+|A^{\lambda}|_{r})
\end{align*}
for q.e. $(s,x)\in D_{T}$. Hence, by Lemma \ref{lm2.3},
$\nu\in\MM_{0,b}^{+}(D_{T})$. From (\ref{eq5.7}) and Proposition
\ref{stw2.3} it follows now that $u$ is a renormalized solution of
the problem (\ref{eq2.C}) with $\mu$ replaced by $\mu+\nu$,
$u\in\mathcal{T}_{2}^{0,1}$ and $u\in L^{q}(0,T;W_{0}^{1,q}(D))$
for $q\in[1,\frac{d+2}{d+1})$. Actually, from the definition of a
solution of reflected BSDE and (\ref{eq5.5}), (\ref{eq5.07}) it
follows that the pair $(u,\nu)$ is a solution of
OP$(\varphi,f+d\mu,h_1)$. Moreover, from (\ref{eq5.5}) and the
fact that for q.e. $(s,x)\in D_T$ the process $Y^{s,x}$ is of
class (D) and $Y^{s,x}\in\mathcal{D}^{q}$ for $q\in(0,1)$ we
conclude that $u$ is of class \mbox{\rm(FD)} and
$u\in\FF\mathcal{D}^{q}$ for $q\in(0,1)$. Furthermore, since
$(u(\bfX),\psi(\bfX))$ is a solution of
BSDE${}_{s,x}(\varphi,D,f_u+d(\mu+\nu))$, it follows from
Proposition \ref{stw3.7} that for q.e. $(s,x)\in D_{T}$,
\begin{equation}
\label{eq5.08} \sigma\nabla u(\bfX)=\psi(\bfX), \quad dt\otimes
P'_{s,x}\mbox{-a.e. on } [0,\zeta_{\tau}]\times\Omega'.
\end{equation}
Hence $\nabla u\in FM^q$ for $q\in(0,1)$, because $Z^{s,x}\in
M^q$, $q\in(0,1)$, for q.e. $(s,x)\in D_T$.
\smallskip\\
(ii) From (\ref{eq3.4}) and (\ref{eq5.5}) it follows that
$u_n\nearrow u$ q.e. on $D_T$, whereas from (\ref{eq3.5}),
(\ref{eq5.6}), (\ref{eq5.08}) and Lemma \ref{lm2.4} it follows
that $\nabla u_n\rightarrow\nabla u$ in measure $m_1$. Since
$u_{n}\le u$ q.e. on $D_{T}$, we have
\begin{align*}
E'_{s,x}\int_{0}^{\zeta_{\tau}}dA_{r}^{\nu_{n}} &\le C(\kappa,T)\Big(
E'_{s,x}|\varphi(\bfX_{T-
\tau(0)})|\\&\quad+E'_{s,x}\int_{0}^{\zeta_{\tau}}|f_{u}(\bfX_{r})|\,dr
+E'_{s,x}\int_{0}^{\zeta_{\tau}}d(|A^{\mu}|_r+|A^{\lambda}|_{r})\Big).
\end{align*}
Hence $\sup_{n\ge1}\|\nu_{n}\|_{TV}<+\infty$ by Lemma \ref{lm2.3},
and consequently, by Remark \ref{uw4.6} and Vitali's theorem,
$u_n\rightarrow u$ in $L^q(0,T;W^{1,q}_0(D))$ for every
$q\in[1,\frac{d+2}{d+1})$. Suppose that for some subsequence,
still denoted by $n$, $\{\nu_n\}$ converges weakly on $D_{T}$ to
some measure $\nu'$. By Remark \ref{uw4.1}, $u_n$ is a
distributional solution of (\ref{eq3.3}) and $u$ is a solution of
(\ref{eq2.C}) with $\mu$ replaced by $\mu+\nu$, i.e for any
$\eta\in C^{\infty}_0(D_T)$,
\[
\int_{D_T}(u_n\frac{\partial\eta}{\partial t}+a\nabla
u_n\cdot\nabla\eta)\,dm_1 =\int_D\varphi\eta(T,\cdot)\,dm
+\int_{D_T}f_{u_n}\eta\,dm_1 +\int_{D_T}\eta\,d(\mu+\nu_n)
\]
and (\ref{eq4.6}) is satisfied with $\mu$ replaced by $\mu+\nu$.
Since $f_{u_n}\rightarrow f_u$, $m_1$-a.e. and by (H3),
$f_v+\kappa(u_1-v)\le f_{u_n}\le f_{u_1}+\kappa(u-u_1)$, applying
the Lebesgue dominated convergence theorem we conclude that
$f_{u_n}\rightarrow f_u$ in $L^1(D_T)$. Since we also know that
$u_n\rightarrow u$ and $\nabla u_n\rightarrow\nabla u$ in
$L^1(D_T)$, it follows that
$\int_{D_T}\eta\,d\nu=\int_{D_T}\eta\,d\nu'$, and hence that
$\nu=\nu'$. Thus $\nu_{n}\rightarrow\nu$ weakly on $D_{T}$, and
the proof is complete.
\end{dow}

\begin{wn}
Assume \mbox{\rm(H1)--(H4)}. Let $(u,\mu)$ be a solution of
\mbox{\rm OP}$(\varphi,f+d\mu,h_{1})$. Then
\[
u=\mbox{\rm quasi-essinf}\{v\ge h_{1},\,m_{1}\mbox{-a.e.}:\,\,
v\mbox{ is a supersolution of \,\rm{PDE}}(\varphi, f+d\mu)\}
\]
q.e. on $D_{T}$.
\end{wn}
\begin{dow}
Follows from Theorem \ref{tw5.5}, Proposition \ref{stw3.7},
Theorem \ref{tw3.1}  and \cite[Lemma 4.9]{Kl:BSM}.
\end{dow}

\begin{wn}
Let $(u,\nu)$ be a solution of
\mbox{\rm OP}$(\varphi, f+d\mu,h_{1})$ and let $h^{*}\in
\FF\mathcal{D}$ be such that $h\le h^{*}\le u$, $m_{1}$-a.e. Then
for q.e. $(s,x)\in D_{T}$,
\begin{align*}
u(s,x)&=\sup_{\sigma\in\mathcal{T}'}E'_{s,x}
\Big(\int_{0}^{\sigma\wedge \zeta_{\tau}}
f_{u}(\bfX_{r})\,dr+\int_{0}^{\sigma\wedge\zeta_{\tau}}dA_{r}^{\mu}\\
&\quad+h^{*}(\bfX_{\sigma})\mathbf{1}_{\{\sigma<\zeta_{\tau}\}}
\mathbf{1}_{\{\sigma<T-\tau(0)\}}
+\varphi(\bfX_{T-\tau(0)})\mathbf{1}_{\{\sigma=T-\tau(0)\}}\Big),
\end{align*}
where $\mathcal{T}'$ denotes the set of all $\{\GG'_t\}$-stopping
times.
\end{wn}
\begin{dow}
Follows from Theorem \ref{tw5.5} and \cite[Lemma 4.9]{Kl:BSM}.
\end{dow}
\medskip

We now turn to the obstacle problem with two barriers.
\begin{tw}
\label{tw5.8} Assume \mbox{\rm(H1)--(H5), (H6$'$)}.
\begin{enumerate}
\item[\rm(i)]There exists a solution $(u,\nu)$ of \mbox{\rm OP}$(\varphi,
f+d\mu,h_{1},h_{2})$ such that $u$ is of class \mbox{\rm(FD)},
$u\in\FF\mathcal{D}^{q}$ for $q\in(0,1)$, $u\in
\mathcal{T}_{2}^{0,1}$, $L^{q}(0,T,W_{0}^{1,q}(D))$ for
$q\in[1,\frac{d+2}{d+1})$ and $\nabla u\in FM^{q}$ for
$q\in(0,1)$.
\item[\rm(ii)]If $A^{\mu}$ is continuous, $h_{1},h_{2}$ are quasi-continuous
and $h_{1}(T,\cdot)\le\varphi\le h_{2}(T,\cdot)$, $m$-a.e. then
$u$ is quasi-continuous.
\item[\rm(iii)] Let $u_{n}$ be a renormalized solution of the problem
\[
\left\{
\begin{array}{l}\frac{\partial u_{n}}{\partial t}+A_{t}u_{n}
=-f_{u_{n}}-\mu-n(u_{n}-h_{1})^{-}+n(u_{n}-h_{2})^{-}, \smallskip\\
u_{n}(T, \cdot)=\varphi, \quad u_{n}(t, \cdot)_{|\partial D}=0,
\quad t\in(0,T).
\end{array}
\right.
\]
Then $u_{n}\rightarrow u$ q.e. on $D_{T}$ and  $u_{n}\rightarrow
u$ in $L^q(0,T;W^{1,q}_0(D))$ for $q\in[1,\frac{d+2}{d+1})$.
\item[\rm(iv)]Let $(\underline{u}_{n},\underline{\beta}_{n})$ be a solution
of $\underline{\mbox{\rm O}}\mbox{\rm P}
(\varphi,\underline{f}_{n}+d\mu,h_{1})$ with
\[
\underline{f}_{n}(t,x,y)=f(t,x,y)-n(y-h_{2}(t,x))^{+}.
\]
Then $\underline{u}_{n}\searrow u$ q.e. on $D_{T}$,
$\underline{u}_{n}\rightarrow u$ in $L^q(0,T;W^{1,q}_0(D))$ for
$q\in[1,\frac{d+2}{d+1})$,
$\underline{\beta}_{n}\ge\underline{\beta}_{n+1}$,
$\underline{\beta}_{n} \searrow \nu^{+}$ setwise,
$\underline{\gamma}_{n}\rightarrow \nu^{-}$ weakly on $D_{T}$,
where $\underline{\gamma}_{n}=n(\underline{u}_{n}-h_{2})^{+}\cdot
m_{1}$.
\end{enumerate}
\end{tw}
\begin{dow}
By Lemma \ref{lm3.1},  Lemma \ref{lm3.2} and \cite[Theorem
6.6]{Kl:BSM}, for q.e. $(s,x)\in D_{T}$ there exists a solution
$(Y^{s,x},Z^{s,x},R^{s,x})$ of RBSDE$_{s,x}(\varphi,D, f+d\mu,
h_{1},h_{2})$ such that $Y^{s,x}$ is of class (D),
$(Y^{s,x},Z^{s,x})\in\mathcal{D}^{q}\otimes M^{q}$ for $q\in(0,1)$
and $R^{s,x}\in\mathcal{V}^{1}$. Moreover, by \cite[Theorem
6.6]{Kl:BSM} again, for q.e. $(s,x)\in D_{T}$,
\begin{equation}
\label{eq3.6} \underline{Y}_{t}^{n,s,x}\searrow Y_{t}^{s,x}, \quad
t\in[0,\zeta_{\tau}],\,P'_{s,x}\mbox{-a.s.},
\end{equation}
\begin{equation}
\label{eq3.7} \underline{Z}^{n,s,x} \rightarrow Z^{s,x}, \quad
dt\otimes P'_{s,x}\mbox{-a.e. on }[0,\zeta_{\tau}]\times\Omega',
\end{equation}
\begin{equation}
\label{eq3.8} \underline{K}_{t}^{n,s,x}\searrow R_{t}^{s,x,+},
\quad t\in[0,\zeta_{\tau}],\, P'_{s,x}\mbox{-a.s.},
\end{equation}
where $(\underline{Y}_{t}^{n,s,x},\underline{Z}^{n,s,x},
\underline{K}_{t}^{n,s,x})$ is a solution of
$\underline{\mbox{R}}$BSDE$_{s,x}(\varphi,D,
\underline{f}_{n}+d\mu,h_{1})$. By Theorem \ref{tw5.5},
$\underline{u}_{n}(\bfX_{t})=\underline{Y}_{t}^{n,s,x}$,
$t\in[0,\zeta_{\tau}]$, $P'_{s,x}$-a.s,
$\sigma\nabla\underline{u}_{n}(\bfX)=\underline{Z}^{n,s,x}$,
$dt\otimes P'_{s,x}$-a.e. on $[0,\zeta_{\tau}]\times\Omega'$ and
$\underline{K}_{t}^{n,s,x}=A_{t}^{\underline{\beta}_{n}}$,
$t\in[0,\zeta_{\tau}]$, $P'_{s,x}$-a.s. Put
$u(s,x)=\limsup_{n\rightarrow+\infty}\underline{u}_{n}(s,x)$,
$(s,x)\in D_{T}$. Then by (\ref{eq3.6}),
\begin{equation}
\label{eq5.13} u(\bfX_{t})=Y_{t}^{s,x}, \quad
t\in[0,\zeta_{\tau}], \,P'_{s,x}\mbox{-a.s.}
\end{equation}
for q.e. $(s,x)\in D_T$. By Lemma \ref{lm2.4} and
(\ref{eq3.7}) there exists a Borel measurable function $\psi$ on
$D_{T}$ such that
\begin{equation}
\label{eq5.14} \psi(\bfX)=Z^{s,x}, \quad dt\otimes
P'_{s,x}\mbox{-a.e. on } [0,\zeta_{\tau}]\times\Omega'
\end{equation}
for q.e. $(s,x)\in D_{T}$. Since
${\underline{\beta}_{n}}-{\underline{\beta}_{n+1}}$ is the Revuz
measure of the AF
$A_{t}^{\underline{\beta}_{n}}-A_{t}^{\underline{\beta}_{n+1}}$
and by (\ref{eq3.8}) the functional is positive, it follows that
the measure ${\underline{\beta}_{n}}-{\underline{\beta}_{n+1}}$ is
positive. Therefore   $\{\underline{\beta}_{n}\}$ is a
nonincreasing sequence and $R_{t}^{s,x,+}=A_{t}^{\nu_{1}}$,
$t\in[0,\zeta_{\tau}]$, $P'_{s,x}$-a.s., where $\nu_{1}$ is a
setwise limit of $\{\underline{\beta}_{n}\}$. Put
\[
A_{t}=u(\bfX_{0})-u(\bfX_{t})-\int_{0}^{t}f_{u}(\bfX_r)\,dr
+A_{t}^{\nu_{1}}+\int_{0}^{t}\psi(\bfX_{r})\,dB_{r},\quad t\ge0.
\]
As in the proof of Theorem \ref{tw5.5} one can show that there
exists a nonnegative smooth measure $\nu_{2}$ on $D_T$ such that
$A=A^{\nu_{2}}$. Let $\nu=\nu^{1}-\nu^{2}$. From the construction
of $\nu^{1}, \nu^{2}$ it follows that $\nu^{1}=\nu^{+}$,
$\nu^{2}=\nu^{-}$ and
\begin{equation}
\label{eq5.15} A^{\nu}_{t}=R_{t}^{s,x}, \quad
t\in[0,\zeta_{\tau}],\, P'_{s,x}\mbox{-a.s.}
\end{equation}
for q.e. $(s,x)\in D_{T}$. By (\ref{eq5.13}) and (\ref{eq5.15}),
\begin{equation}
\label{eq5.16}
u(s,x)=E'_{s,x}\Big(\mathbf{1}_{\{\zeta>T-\tau(0)\}}
\varphi(\bX_{T-\tau(0)}) +\int_{0}^{\zeta_{\tau}}f_u(\bfX_r)\,dr
+\int_{0}^{\zeta_{\tau}}d(A_{r}^{\mu}+A^{\nu}_r)\Big)
\end{equation}
for q.e. $(s,x)\in D_T$. By (6.8) in \cite{Kl:BSM} (see the
beginning of the proof of \cite[Theorem 6.6]{Kl:BSM}) and Theorems
\ref{tw3.1}, \ref{tw5.5} and Lemma \ref{lm2.3},
\begin{equation}
\label{eq5.17} \|\underline{\beta}_{n}\|_{TV}\le
\|\underline{\beta}_{1}\|_{TV}, \quad
\|\underline{\gamma}_{n}\|_{TV}\le \|\delta_{2}^{n}\|_{TV},
\end{equation}
where $\delta_{2}^{n}=n(v_{2}^{n}-h_{2})^{+}\cdot m_{1}$,
$v_{2}^{n}$ is a renormalized solution of PDE$(\phi\vee\varphi,
\underline{f}_{n}+d\lambda_{2}+d\mu)$ with
$\lambda_{2}=f_{v}^{-}\cdot m_{1}+\lambda^{-}+\mu^{-}$ ($\lambda$
is the measure from condition (H6$'$)). From the first inequality
in (\ref{eq5.17}) it follows that $\|\nu^{+}\|_{TV}<+\infty$.
Furthermore, from Theorem \ref{tw5.5} we know that
$\sup_{n\ge1}\|\delta_{2}^{n}\|_{TV}<+\infty$. Therefore
$\{\underline{\gamma}_{n}\}$ is tight, which in fact implies that
$\underline{\gamma}_{n}\rightarrow\nu^{-}$ weakly on $D_{T}$ (see
the reasoning at the end of the proof of Theorem \ref{tw5.5}).
Hence $\|\nu^{-}\|_{TV}<+\infty$, and consequently
$\nu\in\MM_{0,b}(D_T)$. By (6.32) in \cite{Kl:BSM} and Theorem
\ref{tw5.5}, $\underline{u}_n\ge\bar v$, where $\bar v$ is the
first component of the solution of $\overline{\mbox{\rm
O}}\mbox{\rm P} (\varphi,f+d\mu,h_{2})$. Using
(\ref{eq5.13})--(\ref{eq5.16}) and Proposition \ref{stw2.3} one
can deduce now in much the same way as in the proof of Theorem
\ref{tw5.5} that $(u,\nu)$ is a solution of
OP$(\varphi,f+d\mu,h_{1},h_{2})$,  $u$ has the regularity
properties stated in (i) and
\begin{equation}
\label{eq5.18} \sigma\nabla u(\bfX)=\psi(\bfX), \quad dt\otimes
P'_{s,x}\mbox{-a.e. on } [0,\zeta_{\tau}]\times\Omega'
\end{equation}
for q.e. $(s,x)\in D_T$. From what has already been proved,
(\ref{eq3.7}), (\ref{eq5.14}), (\ref{eq5.18}) and Lemma
\ref{lm2.4} we get (iv). Finally, assertions (ii) and (iii) follow
from \cite[Theorem 6.6(ii),(v)]{Kl:BSM}, Lemma \ref{lm2.4},
Remark \ref{uw4.6} and the stochastic representation
(\ref{eq5.13})--(\ref{eq5.15}), (\ref{eq5.18}) of the solution
$(u,\nu)$.
\end{dow}

\begin{wn}
\label{wn5.9} Assume that $A^{\mu}$ is continuous, the barriers
$h_{1},h_{2}$ are quasi-continuous and
$h_{1}(T,\cdot)\le\varphi\le h_{2}(T,\cdot)$, $m$-a.e. Then the
minimality condition \mbox{\rm(c)} in the definition of a solution
of the obstacle problem is equivalent to \mbox{\rm(\ref{eq1.05})}.
\end{wn}
\begin{proof}
Since $u,h_1,h_2$ are quasi-continuous, from condition (c) it
follows that
\begin{equation}
\label{eq5.8}
E'_{s,x}\int^{\zeta_{\tau}}_0(u-h_1)(\bfX_r)\,dA^{\nu^+}_r
=E'_{s,x}\int^{\zeta_{\tau}}_0(h_2-u)(\bfX_r)\,dA^{\nu^-}_r=0.
\end{equation}
Hence
\begin{align}
\label{eq5.9}
&\int^T_s\!\!\int_D(u-h_1)(t,y)p_D(s,x,t,y)\,d\nu^+(t,y)\nonumber \\
&\qquad=\int^T_s\!\!\!\int_D(h_2-u)(t,y)p_D(s,x,t,y)\,d\nu^-(t,y)=0,
\end{align}
which implies (\ref{eq1.05}), because $p_D(s,x,\cdot,\cdot)$ is
positive on $(s,T]\times D$. Conversely, assume that
(\ref{eq1.05}) is satisfied. Then (\ref{eq5.9}), and consequently
(\ref{eq5.8}) is satisfied. Of course (\ref{eq5.8}) implies (c).
\end{proof}
\medskip

Let us note that in general, even in the case of one quasi-l.s.c.
or quasi-u.s.c. reflecting barrier, the integrals in
(\ref{eq1.05}) may be strictly positive (see \cite[Example
5.5]{Kl:SPA}).

\begin{wn}
Assume \mbox{\rm(H1)--(H4)}. Let $(u,\nu)$ be a solution of
\mbox{\rm OP}$(\varphi,f+d\mu,h_{1},h_{2})$. Then
\[
u=\mbox{\rm quasi-}\essinf\{v\ge h_{1},\, m_{1}\mbox{-a.e.}:\,\,
v\mbox{ is a supersolution of
\,\rm{PDE}}(\varphi,f+d\mu-d\nu^{-})\}
\]
q.e. on $D_T$.
\end{wn}
\begin{dow}
Follows from Theorem \ref{tw5.8} and \cite[Lemma 4.9]{Kl:BSM}.
\end{dow}

\begin{wn}
Let $(u,\nu)$ be a solution of
\mbox{\rm OP}$(\varphi,f+d\mu,h_{1},h_{2})$ and let $h^{*}_{1},
h^{*}_{2}\in \mathcal{FD}$ be such that $h_{1}\le h^{*}_{1}\le
u\le h^{*}_{2}\le h_{2}$, $m_{1}$-a.e. Then
\begin{align*}
u(s,x)&=\esssup_{\sigma\in\mathcal{T}'}\essinf_{\delta\in\mathcal{T}'}
E'_{s,x}\Big(\int_{0}^{\sigma\wedge\delta\wedge\zeta_{\tau}}
f_{u}(\mathbf{X}_{r})\,dr
+\int_{0}^{\sigma\wedge\delta\wedge\zeta_{\tau}}dA^{\mu}_{r}\\
&\quad+h^{*}_{1}(\mathbf{X}_{\delta})
\mathbf{1}_{\{\delta\le\sigma< T-\tau(0)\}}
\mathbf{1}_{\{\delta<\zeta_{\tau}\}}
+h^{*}_{2}(\mathbf{X}_{\sigma})\mathbf{1}_{\{\sigma<\delta\}}
\mathbf{1}_{\{\sigma<\zeta_{\tau}\}} \\
&\quad +\varphi(\mathbf{X}_{T-\tau(0)})
\mathbf{1}_{\{\sigma=\delta=T-\tau(0)\}}\Big)
\end{align*}
q.e. on $D_T$, where $\mathcal{T}'$ is the set of all
$\{\GG'_t\}$-stopping times.
\end{wn}
\begin{dow}
Follows from Theorem \ref{tw5.8} and \cite[Proposition 3.1]{Lepeltier}.
\end{dow}
\medskip\\
{\bf Acknowledgements}
\medskip\\
The first author was supported by  NCN grant no.
2012/07/D/ST1/02107, the second author was supported by NCN grant
no. 2012/07/B/ST1/03508.


\begin{thebibliography}{99}

\bibitem{Aronson}
{\sc D.G. Aronson}, Non-Negative Solutions of Linear Parabolic
Equations. {\em Ann. Sc. Norm. Super. Pisa} {\bf 22} (1968)
607--693.

\bibitem{BL}
{\sc A. Bensoussan, J.-L. Lions}, Applications of Variational
Inequalities in Stochastic Control. North-Holland, Amsterdam,
1982.

\bibitem{BB}
{\sc L. Beznea, N. Boboc},  Potential Theory and Right Processes.
Kluwer Academic Publishers, Dordrecht, 2004.

\bibitem{BBGGPV}
{\sc P. B\'enilan, L. Boccardo, T. Gallou\"et, R. Gariepy, M.
Pierre, J.-L. Vazquez}, An $L^{1}$-theory of existence and
uniqueness of solutions of nonlinear elliptic equations. {\em Ann.
Scuola Norm. Sup. Pisa Cl. Sci.} {\bf 22} (1995) 241--273.

\bibitem{BDGO}
{\sc L. Boccardo, A. Dall'Aglio, T. Gallou\"et, L. Orsina},
Nonlinear Parabolic Equations with Measure Data.  {\em J. Funct.
Anal.} {\bf 147} (1997) 237--258.

\bibitem{BG}
{\sc L. Boccardo, T. Gallou\"et}, Non-linear Elliptic and
Parabolic Equations Involving Measure Data. {\em J. Funct. Anal.}
{\bf 87} (1989) 149--169.

\bibitem{DP}
{\sc J. Droniou, A. Prignet}, Equivalence between entropy and
renormalized solutions for parabolic equations with smooth measure
data.  {\em NoDEA Nonlinear Differential Equations Appl.} {\bf 14}
(2007) 181--205.

\bibitem{DPP}
{\sc J. Droniou, A. Porretta, A. Prignet}, Parabolic Capacity and
Soft Measures for Nonlinear Equations. {\em Potential Anal.} {\bf
19} (2003) 99--161.

\bibitem{Fukushima}
{\sc M. Fukushima, Y. Oshima, M. Takeda}, Dirichlet Forms and
Symmetric Markov Processes. De Gruyter Studies in Mathematics 19,
Walter de Gruyter, New York, 1994.

\bibitem{GS}
{\sc R.K. Getoor, M.J. Sharpe},  Naturality, standardness, and
weak duality for Markov processes. {\em Z. Wahrsch. verw. Gebiete}
{\bf 67} (1984) 1--62.


\bibitem{Kl:SPA}
{\sc T. Klimsiak}, Reflected BSDEs and the obstacle problem for
semilinear PDEs in divergence form. {\em Stochastic Process.
Appl.} {\bf 122} (2012) 134--169.


\bibitem{Kl:BSM}
{\sc T. Klimsiak}, BSDEs with monotone generator and two irregular
reflecting barriers. {\em Bull. Sci. Math.}  {\bf 137} (2013)
268--321.

\bibitem{Kl:PA2}
{\sc T. Klimsiak}, Cauchy problem for semilinear parabolic
equation with time-dependent obstacles: a BSDEs approach. {\em
Potential Anal.} {\bf 39} (2013) 99--140.

\bibitem{K:JTP}
{\sc T. Klimsiak}, On time-dependent functionals of diffusions
corresponding to divergence form operators. {\em J. Theor.
Probab.} {\bf 26} (2013) 437--473.

\bibitem{KR}
{\sc T. Klimsiak, A. Rozkosz}, Dirichlet forms and semilinear
elliptic equations with measure data. {\em J. Funct. Anal.} {\bf
265} (2013) 890--925.

\bibitem{Lejay}
{\sc A. Lejay}, A probabilistic representation of the solution of
some quasi-linear PDE with a divergence form operator. Application
to existence of weak solutions of FBSDE. {\em Stochastic Process.
Appl.} {\bf 110} (2004) 145--176.

\bibitem{Leone}
{\sc C. Leone}, Existence and uniqueness of solutions for
nonlinear obstacle problems with measure data. {\em Nonlinear
Anal.} {\bf 43} (2001) 199--215.

\bibitem{Lepeltier}
{\sc J.P. Lepeltier, M. Xu}, Reflected backward stochastic
differential equations with two RCLL barriers. {\em ESAIM Probab.
Stat.} {\bf 11} (2007) 3--22.

\bibitem{Meyer}
{\sc P.-A. Meyer}, Fonctionnelles multiplicatives et additives de
Markov.  {\em Ann. Inst. Fourier} {\bf 12} (1962) 125--230.

\bibitem{MP}
{\sc F. Mignot, J.P. Puel}, In\'equations d'\'evolution
paraboliques avec convexes d\'ependant du temps. Applications aux
in\'equations quasi-variationnelles d'\'evolution. {\em Arch.
Ration. Mech. Anal.} {\bf 64} (1977) 59--91.

\bibitem{Oshima}
{\sc Y. Oshima}, Some properties of Markov processes associated
with time dependent Dirichlet forms. {\em Osaka J. Math.} {\bf 29}
(1992) 103--127.

\bibitem{Oshima2}
{\sc Y. Oshima}, Time-dependent Dirichlet forms and related
stochastic calculus. {\em Infin. Dimens. Anal. Quantum Probab.
Relat. Top.} {\bf 7} (2004)  281--316.

\bibitem{Oshima1}
{\sc Y. Oshima},  Semi-Dirichlet Forms and Markov Processes.
Walter de Gruyter, Berlin, 2013.

\bibitem{Petitta}
{\sc F. Petitta}, Renormalized solutions of nonlinear parabolic
equations with general measure data.  {\em Ann. Mat. Pura Appl.}
{\bf 187} (2008) 563--604.

\bibitem{PPP}
{\sc F. Petitta, A.C. Ponce, A. Porretta}, Diffuse measures and
nonlinear parabolic equations. {\em J. Evol. Equ.} {\bf 11} (2011)
861--905.

\bibitem{Pierre1}
{\sc M. Pierre}, Problemes d'Evolution avec Contraintes
Unilat\'erales et Potentiel Paraboliques. {\em Comm. Partial
Differential Equations} {\bf 4} (1979) 1149--1197.

\bibitem{Pierre}
{\sc M. Pierre}, Representant Precis d'Un Potentiel Parabolique,
Seminaire de Theorie du Potentiel, Paris, No. 5. {\em Lecture
Notes in Math.} {\bf 814} (1980) 186--228.

\bibitem{Revuz}
{\sc D. Revuz}, Mesures associees aux fonctionnelles additives de
Markov I. {\em Trans. Amer. Math. Soc.} {\bf 148} (1970) 501--531.

\bibitem{R:PTRF}
{\sc A. Rozkosz}, Backward SDEs and Cauchy problem for semilinear
equations in divergence form. {\em Probab. Theory Related Fields}
{\bf 125} (2003) 393--407.

\bibitem{R:Stochastics}
{\sc A. Rozkosz}, On the Feynman-Kac representation for solutions
of the Cauchy problem for parabolic equations in divergence form.
{\em Stochastics}  {\bf 77} (2005) 297--313.

\bibitem{RT}
{\sc F. Russo, G. Trutnau}, About a construction and some analysis
of time inhomogeneous diffusions on monotonely moving domains.
{\em J. Funct. Anal.} {\bf 221} (2005) 37--82.

\bibitem{Sharpe}
{\sc M. Sharpe}, General Theory of Markov Processes. Academic
Press, Boston, 1988.
\end{thebibliography}
\end{document}